\documentclass[review]{elsarticle}

\usepackage{setspace}

\usepackage[margin=1in]{geometry}

\usepackage{graphicx}
\usepackage{tikz}
\usepackage{pgfplots}
\usepackage{amsmath,amsthm,amssymb}
\usepackage{derivative}
\usepackage{framed}
\usepackage{algorithm}
\usepackage{algorithmicx}
\usepackage{algpseudocode}
\usepackage{url}
\usepackage{booktabs}
\usepackage{xcolor}
\usepackage{multirow}
\usepackage{bm}

\usepackage{xcolor}
\newcommand{\modified}[1]{{\color{black}#1}}

\usepackage{lineno,hyperref}
\modulolinenumbers[5]

\journal{Preprint submitted to arXiv}

\biboptions{authoryear}

\newtheorem{definition}{Definition}

\newtheorem{lemma}{Lemma}

\begin{document}

\begin{frontmatter}

\title{Mixed-Integer Linear Programming Approximations for the\\Stochastic Knapsack\tnoteref{mytitlenote}}
\tnotetext[mytitlenote]{This material is based upon works supported by the Science Foundation Ireland under Grant No. 12/RC/2289-P2 which is co-funded under the European Regional Development Fund.}

\author[rossi]{Roberto Rossi\corref{mycorrespondingauthor}}
\cortext[mycorrespondingauthor]{Corresponding author}
\ead{roberto.rossi@ed.ac.uk}

\author[steve]{Steven D. Prestwich}
\ead{s.prestwich@ucc.ie}

\author[tarim]{S. Armagan Tarim}
\ead{armagan.tarim@hacettepe.edu.tr}

\address[rossi]{Business School, University of Edinburgh, Edinburgh, UK}
\address[steve]{School of Computer Science \& IT, University College Cork, College Rd, Cork, Ireland}
\address[tarim]{Institute of Informatics, Hacettepe University, Ankara, Türkiye}

\begin{abstract}
 \modified{We develop mathematical programming approximations} to tackle the stochastic knapsack problem. 
In this problem, the decision maker considers items for which either weights or values, or both, are random. The aim is to select a subset of these items to be included into their knapsack. 
We study both static and dynamic variants of this problem: in the static setting, the decision about which items should be included in the knapsack is taken at the outset, before any random item value or weight is revealed; in the dynamic setting, items are received sequentially, and the decision about a particular item is made by taking into account previously observed values and weights. 
The knapsack has a given capacity, and if the total realised weight exceeds this capacity then a penalty cost is incurred for each unit of excess capacity utilised.
The goal is to maximise the expected net profit. 
We tackle the case of normally distributed item weights and we show that our approach extends to the case in which item weights are correlated and follow a multivariate normal distribution. In addition, we show our approach represents an effective heuristic for the case in which item weights follow generic probability distributions. In an extensive computational study we demonstrate that our models are near-optimal and more scalable than other state-of-the-art approaches.
\end{abstract}


\end{frontmatter}


\section{Introduction}\label{introduction}

The knapsack problem was first discussed by \cite{05cf96e1-34b5-3cd9-b02c-2a84c069f688}; since Dantzig's pioneering work, the problem has been intensively studied because of its relevance to both the theory and practice of combinatorial optimisation \citep[see e.g.][]{martello1990knapsack,Cacchiani2022}.
In its original formulation, the decision maker considers several items with predetermined weights and values. Their goal is to select a subset of these items that maximises the knapsack total value, whilst ensuring the knapsack weight stays below a specified capacity. 
In practice, it is often the case that item weights and values are unknown at the time item insertion decisions are made. Stochastic knapsack problems were introduced to address the situation in which either item weights or values, or both, are random. Stochastic knapsack problems can be categorised into two classes: ``static'' and ``dynamic.'' In the static setting, the decision about which items should be included in the knapsack is taken at the outset, before any item value or weight is revealed. In the dynamic setting, items are received sequentially, and the decision about a particular item is taken by taking into account previously observed values and weights. When item weights and values are random, there are several possible alternative ways of modelling what happens if the knapsack capacity is exceeded. 
In this work, we follow the established framework discussed in works such as \citep{Kleywegt2002,Merzifonluolu2011,Fortz2012}: the knapsack features a given capacity, and if the total realised weight exceeds this capacity, a penalty cost is incurred for each unit of excess capacity utilised; moreover, a salvage value can be claimed for each unit of capacity that remains unused. The goal is to maximise the expected net profit of the knapsack. 

We make the following contribution to the literature on the stochastic knapsack problem:
\begin{itemize}
\item \modified{We develop mathematical programming approximations} to tackle the Static Stochastic Knapsack Problem (SSKP) under normally distributed item weights; \modified{our models  do not require any dedicated coding and can be implemented directly in standard mathematical programming packages without specialised callbacks or decomposition code.} To model expected capacity shortfall, we adopt a piecewise linearisation strategy and we derive explicit conditions on the numbers of segments that guarantee a given absolute optimality tolerance.
\item Our approach extends to the case in which item weights are correlated and follow a multivariate normal distribution; to the best of our knowledge, ours is the first quadratically constrained programming formulation for SSKP under multivariate normal item weights.
\item In an extensive computational study, we compare our SSKP heuristic against four state-of-the-art baselines from the literature \citep{Kleywegt2002, Merzifonluolu2011, Fortz2012, Range2018} and
we find that our approach is more scalable, whilst remaining near-optimal. In addition, we demonstrate empirically that our approach can be used to generate near-optimal solutions to instances in which item weights follow generic probability distributions.
\item Finally, we investigate a dynamic setting in which items are considered sequentially, and a decision on an item inclusion may depend on previous item weight realisations; we deploy our SSKP heuristic in a receding horizon control framework \citep{Kwon2005-fs}, and compare its performance against optimal solutions obtained via stochastic dynamic programming for uncorrelated and correlated item weights; for the latter case, we empirically estimate the cost of ignoring item correlation and show that this is non-negligible. 
\modified{Additionally, we provide a statistical optimality gap certificate that enables empirical assessment of solution quality even when stochastic dynamic programming is computationally intractable.}
\end{itemize}

\newpage

\modified{This work is organised as follows. In Section \ref{sec:literature} we survey related literature. In Section \ref{sec:sskp} we introduce the SSKP. In Section \ref{sec:piecewise} we provide relevant formal background on the piecewise linearisation of the first order loss function. In Section \ref{sec:milp_sskp}, we illustrate how linearisation strategies can be leveraged to solve the SSKP under weights that follow a (multivariate) normal distribution. In Section \ref{sec:dskp}, we introduce the Dynamic Stochastic Knapsack Problem (DSKP), which generalises the stochastic knapsack problem to a dynamic setting, and we deploy mathematical programming approaches developed for the SSKP in a receding horizon control framework to tackle the DSKP. In Section \ref{sec:experiments} we present our computational study. In Section \ref{sec:conclusions} we draw conclusions.}

\section{Literature review}\label{sec:literature}

The knapsack problem is one of the most active and investigated areas of research in combinatorial optimisation. A comprehensive survey of this area, presented in \citep{Cacchiani2022}, provides a valuable overview, which however does not cover stochastic knapsack problems. In this section, we aim to complement the former study and survey existing lines of research on the stochastic knapsack problem in both its static and dynamic variants. The goal is to position our contribution within this broad stream of literature.

The definition of the deterministic knapsack problem is well-understood in the literature. Conversely, there are multiple ways in which uncertainty may affect a knapsack problem. This has led to the development of several different variants of the so-called stochastic knapsack problem. 
As we have previously noted, stochastic knapsack problems can be classified as ``static'' or  ``dynamic''' on the basis of the timing of item insertion decisions and observations of random variables. This leads to two variants of the problem under scrutiny: the SSKP, and the DSKP.
We contribute to streams of literature encompassing both problems, which we survey in the following paragraphs.
 
Early works on the SSKP problem \citep[see e.g.][]{Morton1998} focused on the case in which the weight of each item and the knapsack capacity are deterministic, while item revenues are random with known distribution. One of the first works investigating the case in which item revenues and knapsack capacity are known and item weights are random is \citep{Kleywegt2002}, in which the authors utilised this variant of the SSKP to operationalise their Sample Average Approximation (SAA) framework. SAA was in fact directly inspired by \citep{Morton1998}, as discussed by \citeauthor{Kleywegt2002} in their introduction. The SSKP as formulated in \citep{Kleywegt2002} is the first occurrence in the literature of a model resembling the SSKP variant investigated in the present study. \cite{Aral2009} consider a subclass of \citeauthor{Kleywegt2002}'s SSKP model in which random item weights follow a Poisson distribution; for this setting, they discuss a polynomial-time solution for the continuous relaxation of this problem, and a customised branch-and-bound algorithm for the binary version of the problem. \cite{Kosuch2009} focus on the case in which weights are normally distributed random variables with known mean and variance, whilst capacity and item revenues remain deterministic; they investigate continuous relaxations solved using a stochastic gradient method to provide upper bounds in a branch-and-bound framework. \cite{Merzifonluolu2011} consider the same problem setting and provide an exact solution method as well as some structural results for a continuous relaxation of the problem; they then proceed and develop a customised branch-and-bound algorithm and a fast heuristic approach for obtaining an optimal binary solution. Moreover, while the knapsack capacity is considered deterministic in their study, \citeauthor{Merzifonluolu2011} show that problems with random capacity, as well as problems in which capacity is a decision variable subject to unit costs, fall within this class of problems. A recent work by the same authors extends this investigation by considering risk aversion \citep{Merzifonluoglu2021}. 
\cite{Fortz2012} discuss complexity results --- showing that the SSKP is weakly NP-hard --- and apply the LP/NLP algorithm of \cite{Quesada1992} to solve the problem. \cite{Chen2015} investigate unimodality and monotonicity properties of the expected knapsack return under the assumption that all random weight distribution have decreasing reversed hazard rate; under this assumption, they develop a heuristic search algorithm similar to simulated annealing. Finally, \cite{Range2018} consider a variant of the SSKP that penalises the expected knapsack overfilling and, in addition, enforces a chance constraint on the overfilling; by leveraging the central limit theorem, the problem is approximated as a shortest path problem with resource constraints and solved via dynamic programming. Our work contributes to this stream of literature by providing a new heuristic, based on mathematical programming, for tackling the SSKP as formulated in works such as \citep{Kleywegt2002,Merzifonluolu2011,Fortz2012} under normal weights. Our approach is computationally appealing and can be directly modelled and solved by off-the-shelf optimisation packages. 
 
Early works on the DSKP assume that objects arrive to and depart from the knapsack at random times \citep[see e.g.][]{Ross1989}; these works were  motivated by the problem of accepting and blocking calls to a circuit-switched telecommunication system which supports a variety of traffic types (e.g. voice, video). Items are generally assumed to arrive according to a Poisson process, and to feature a demand for a limited resource (the knapsack capacity) and an associated reward. Item resource requirements would become known at the time of the item’s arrival, and an item could be accepted or rejected. If accepted, the item would attract a reward; if rejected, it would incur a penalty. In some works \citep[see e.g.][]{Kleywegt1998,Chen2014} it is up to the decision maker to decide when to stop the problem and collect a terminal value --- in particular, \cite{Chen2014} assumes that an overflow event would result in a complete loss of the entire knapsack value. In others, the problem stops when an insertion is attempted with insufficient remaining capacity, with the final overflowing item not contributing to the knapsack value \citep[see e.g.][]{Dean2008,Blado2016,Blado2019,Blado2020}. All these problems substantially differ from the definition of DSKP we adopt in our work, which is the natural dynamic extension of the SSKP as discussed in works such as \citep{Kleywegt2002,Merzifonluolu2011}. In this extension, a finite stream of items is sequentially considered for inclusion, and a knapsack overfill incurs a per unit penalty cost. We introduce a novel approach, based on mathematical programming and receding horizon control \citep{Kwon2005-fs}, for tackling this DSKP variant under random weights following a generic distribution.

\newpage

We are not aware of any work in the literature tackling the SSKP or the DSKP under items featuring correlated, as opposed to independently distributed, random weights. Our models  extends --- for both the SSKP and the DSKP --- to the case of multivariate normally distributed weights. In the case of the DSKP, this leads to a stochastic model in which future item weights depend on observed weights of items previously considered. This is a class of problems that is seldom investigated in the literature due to its complexity. 

An interesting classification of stochastic knapsack problems can be drawn, following \cite{Lyu2022}, on the basis of three typical ways that the information is revealed to the decision maker. Accordingly, the SSKP may be categorised as ``anticipative:'' a problem in which the allocation sequence is determined before random variables are observed; the DSKP may be categorised as ``adaptive:'' a problem in which information on the actual weight of an item is revealed immediately after the decision to select this item; while the traditional knapsack problem may be categorised as ``responsive:'' a problem in which information on the actual weight of each item is revealed at the start of the allocation process, and in which the allocation policy is constructed after knowing all item weights. A further distinction can then be made within the adaptive class: problems in which weight forecasts for items not yet observed are updated on the basis of already observed item weights, and problems in which random item weights are independent; as we will show, our framework extends to the forecast update case, which to the best of our knowledge has not been considered before in the literature.

Finally, it should be noted that there are several other variants of the stochastic knapsack problem in which different modelling assumptions are made in relation to how uncertainty is accounted in the model; for instance, 
\cite{Goyal2010} investigated the chance-constrained knapsack problem, in which the goal is to select a set of items such that profit is maximised and the probability of the total knapsack size exceeding a bound is at most a given threshold. To keep our discussion compact, we have excluded cognate streams of literature in which the modelling assumptions fundamentally differ from ours. 

\section{The Static Stochastic Knapsack}\label{sec:sskp}

We introduce the SSKP by following the established problem setup discussed in the literature; our notation is summarised in Table \ref{tab:symbols}.

We consider a set of $n$ items, each of which features a random revenue $\pi_i$ per unit of weight and a random weight $\omega_i$. The capacity of the knapsack is $C$. If the total realised weight of selected items exceeds $C$, a shortage cost $\hat{c}$ is charged for each unit of excess capacity. If the total realised weight of selected items is below $C$, a salvage value $s<\hat{c}$ per unit of capacity remaining is received. The goal is to select a set of items that maximizes the expected total profit, which comprises the total value of selected items, plus expected salvage value, minus expected capacity shortage cost.

\begin{table}
\modified{
\centering
\begin{tabular}{ll}
$n$			&number of items\\
$i$			&item index\\
$\pi_i$		&item $i$ random revenue per unit of weight, $\bm{\pi}$ denotes the associated $n$-dimensional vector\\
$\omega_i$	&random weight of item $i$, $\bm{\omega}$ denotes the associated $n$-dimensional vector\\
$C$			&knapsack capacity\\
$c$			&shortage cost per unit of capacity\\
$s$			&salvage value per unit of capacity\\
$x_i$		&a binary decision variable set to one if item $i$ is included in the knapsack,\\
			&$\bm{x}$ denotes the associated $n$-dimensional vector\\
$\bm{\varphi}$	&$\bm{\varphi}\triangleq(\bm{\pi}-s\mathbf{1}_n)\circ\bm{\omega}$
\end{tabular}
\caption{Table of symbols}
\label{tab:symbols}
}
\end{table}

\newpage

\modified{
Let $x_i$ be a binary decision variable that is set to one if item $i$ is included in the knapsack; moreover, $\bm{x}'=(x_1,\ldots,x_n)$, $\bm{\pi}'=(\pi_1,\ldots,\pi_n)$, and $\bm{\omega}'=(\omega_1,\ldots,\omega_n)$, where $'$ denotes the vector transpose operator. The SSKP can be formulated as follow:
\begin{align*}
\text{maximize} \quad & \mbox{E}[(\bm{\pi}\circ\bm{\omega})']\bm{x} - \hat{c}\mbox{E}[\max(\bm{\omega}'\bm{x} - C,0)] + s\mbox{E}[\max(C-\bm{\omega}'\bm{x},0)] \\
\text{subject to:} \quad & 
	\bm{x} \in \{0,1\}^n,
\end{align*}
in which $\circ$ denotes the Hadamard product, and $\mbox{E}[\cdot]$ is the expected value operator.
We introduce the following well-known result from stochastic programming.
\begin{lemma}[see e.g. \cite{lawrencesnyder2011}, p. 664, C.37]\label{lemma:loss_complementary} 
Let $y$ be a scalar and $\xi$ a random variable,
\[\mbox{\em E}[\max(y-\xi,0)]=(y-\mbox{\em E}[\xi])+\mbox{\em E}[\max(\xi-y,0)].\]
\end{lemma}
\noindent
We use the result in Lemma \ref{lemma:loss_complementary} to rewrite the SSKP objective function as follows. 

Let $\mathbf{1}'_n\triangleq(\underbrace{1,1,\ldots,1}_{n})$ denote a row vector of $n$ ones, we rewrite
\begin{align*}
\text{maximize} \quad & \mbox{E}[(\bm{\pi}-s\mathbf{1}_n)\circ\bm{\omega}]'\bm{x} - (\hat{c}-s)\mbox{E}[\max(\bm{\omega}'\bm{x} - C,0)] + sC \\
\text{subject to: } \quad & 
	\bm{x} \in \{0,1\}^n.
\end{align*}
We ignore the constant term $sC$, let $\bm{\varphi}\triangleq(\bm{\pi}-s\mathbf{1}_n)\circ\bm{\omega}$, and define $c\triangleq \hat{c}-s$, thus obtaining
\begin{align*}
\text{maximize} \quad &\mbox{E}[\bm{\varphi}']\bm{x} - c\mbox{E}[\max(\bm{\omega}'\bm{x} - C,0)] \\
\text{subject to:} \quad &
	\bm{x} \in \{0,1\}^n.
\end{align*}
In what follows we shall concentrate on this latter formulation. $\mbox{E}[\max(\bm{\omega}'\bm{x} - C,0)]$, which appears in the objective function, is the so-called ``first order loss function.'' In the next section, we provide relevant formal background on piecewise linearisation strategies that can be applied to the first order loss function.
}

\section{Formal background on piecewise linearisation of the first order loss function}\label{sec:piecewise}

\begin{definition}[First order loss function]
Let $\omega$ be a random variable; the first order loss function is defined as
\[\mathcal{L}(y,\omega)\triangleq \mbox{\em E}[\max(\omega - y,0)].\]
\end{definition}

$\mathcal{L}(y,\omega)$ is non-linear and cannot easily be embedded in mixed-integer linear programming (MILP) models. To overcome this issue, we leverage two well-known  inequalities from stochastic programming: Jensen's
and Edmundson-Madanski's inequalities \cite[see][p. 167--168]{citeulike:695971}, which are applicable because $\mathcal{L}(y,\omega)$ is convex \cite[see][Section C.3.1]{lawrencesnyder2011} in $y$ regardless of the distribution of $\omega$.

Let $g_\omega:\Omega \rightarrow \mathbb{R}_{\geq 0}$ be the probability density function of $\omega$.\footnote{Note that we illustrate the case of a continuous random variable $\omega$; extension to discrete random variables is straightforward.} Consider a partition of support $\Omega$ into $W$ disjoint intervals $\Omega_1,\ldots,\Omega_W$. We define, for all $i=1,\ldots,W$, 
\[p_i\triangleq\Pr(\omega \in \Omega_i)=\int_{\Omega_i} g_\omega(t)\mbox{d}t, \qquad \mbox{E}[\omega|\Omega_i]\triangleq p_i^{-1}\int_{\Omega_i} tg_\omega(t)\mbox{d}t.\] 

When applied to the first order loss function, the lower bound in \cite[][Section 8.2, Theorem 1]{Birge2011-oa}
\[\mathcal{L}_{\textup{lb}}(y,\omega)\triangleq \sum_{i=1}^W p_i\max(\mbox{\em E}[\omega|\Omega_i]-y,0)\leq \mathcal{L}(y,\omega)\]
is a piecewise linear function with $W+1$ segments \citep[][Lemma 10]{Rossi2014}; the 
$i$-th segment of $\mathcal{L}_{\textup{lb}}(y,\omega)$ is
\begin{equation}\label{eq:piecewise_lb}
\mathcal{L}^i_{\textup{lb}}(y,\omega)\triangleq \sum_{k=i}^W p_k\mbox{\em E}[\omega|\Omega_k]-y\sum_{k=i}^W p_k,~~~~~
\mbox{\em E}[\omega|\Omega_i]\leq y\leq \mbox{\em E}[\omega|\Omega_{i+1}],
\end{equation}
where $i=1,\ldots,W$; and the $W+1$ segment is $x=0$, for $\mbox{\em E}[\omega|\Omega_W]\leq y \leq \infty$. This lower bound is a direct application of Jensen's inequality.

Edmundson-Madanski's upper bound, $\mathcal{L}_{\textup{ub}}(y,\omega)\triangleq \mathcal{L}_{\textup{lb}}(y,\omega)+e_W\geq \mathcal{L}(y,\omega)$, is obtained by adding a value $e_W$, which denotes the maximum approximation error of the Jensen's bound; for any given $W\in\mathbb{N}$, $e_W$ can be easily obtained by checking the linearisation error at breakpoints of the piecewise linear approximation. 

This discussion assumes that a partition of the support $\Omega$ is given, but the quality of the approximation closely depends on the structure of this partitioning. 
\cite{citeulike:12820831} and \cite{Imamoto2008} discussed how to obtain an optimal partitioning of the support under a framework that minimises the maximum approximation error. In essence, one must find parameters ensuring that approximation errors at piecewise function breakpoints are all equal.  
These results can be directly applied to compute optimal linearisation parameters for the loss function of a standard normally distributed random variable. 
When the number of partitions increase, the bound converges to the true function, therefore the approximation error can in principle be made as small as necessary.

\section{MILP-based solution methods for the SSKP}\label{sec:milp_sskp}

\modified{In this section, we present our MILP reformulation and approximation framework to solve the SSKP under normally distributed item weights. }

\subsection{Normally distributed item weights}\label{sec:normal}

When $\omega$ is normally distributed, well-known standardisation techniques can be leveraged to ensure we only need to use a precompiled set of linearisation parameters for the standard normal random variable. The standardisation approach relies on the following lemma.
\begin{lemma}[see e.g. \cite{lawrencesnyder2011}, p. 664, C.31]\label{lemma:standardization}
Let $\zeta$ be a normal random variable with mean $\mu$ and standard deviation $\sigma$
\[\mathcal{L}(y,\zeta)=\sigma \mathcal{L}\left(\frac{y-\mu}{\sigma},Z\right),\]
where $Z$ is the standard normal random variable.
\end{lemma}
Let $\dot{p}_i$, $\dot{\mbox{E}}[Z|\Omega_i]$ and $\dot{e}_W$ denote optimal linearisation parameters\footnote{We emphasise only dependency of $\dot{e}_W$ on $W$, albeit all linearisation parameters depend on $W$.} that minimise the maximum approximation error with respect to $\mathcal{L}(y,Z)$.
Then, by combining \eqref{eq:piecewise_lb} and Lemma \ref{lemma:standardization}, we obtain the following result.
\begin{lemma}\label{eq:standardized_piecewise_lb}
The $i$-th segment of $\mathcal{L}_{\textup{lb}}(x,\zeta)$ is
\begin{equation*}
\mathcal{L}^i_{\textup{lb}}(y,\zeta)\triangleq \sigma \left(\sum_{k=i}^W \dot{p}_k\dot{\mbox{\em E}}[Z|\Omega_k]-\frac{(y-\mu)}{\sigma}\sum_{k=i}^W \dot{p}_k\right),~~~~~
\dot{\mbox{\em E}}[Z|\Omega_i]\leq y\leq \dot{\mbox{\em E}}[Z|\Omega_{i+1}],
\end{equation*}
where $i=1,\ldots,W$; the $W+1$ segment is $y=0$, for $\dot{\mbox{\em E}}[Z|\Omega_W]\leq y \leq \infty$.
\end{lemma}

This result, together with $\mathcal{L}_{\textup{ub}}$ presented in Section \ref{sec:piecewise}, can be leveraged to obtain compact and yet very effective mathematical programming formulations to compute upper and lower bounds for the SSKP. In addition to being competitive with the state-of-the-art in terms of computational time, our formulations are more convenient than those presented in \citep{Merzifonluolu2011} --- as we do not need to deal with a non-linear continuous relaxation that is neither convex nor concave --- and are also more general, in the sense that they extend to the case of a multivariate normal distribution; they are also more convenient than \citep{Fortz2012}, as they avoid callbacks.

Consider a vector of normally distributed random weights $\bm{\zeta}\triangleq(\zeta_1,\ldots,\zeta_n)$. Let $\bm{\mu}\triangleq(\mbox{E}[\zeta_1],\ldots,\mbox{E}[\zeta_n])$ be the vector of mean weights, and $\bm{\sigma^2}\triangleq(\sigma_1^2,\ldots,\sigma_n^2)$ be the vector of item weight variances. The SSKP can be approximated via the following deterministic mixed-integer nonlinear program, where $M$, $V$, and $S$ are the expected knapsack mean, variance, and standard deviation, respectively; \eqref{eq:pw_ub} directly follows from Lemma \ref{eq:standardized_piecewise_lb} after simplifying for $\sigma$, and $P$ represents the expected capacity shortfall. 

\modified{
\resizebox{0.95\linewidth}{!}{
\begin{minipage}{\linewidth}
\begin{align}
  \text{maximize} \quad & \mbox{E}[\bm{\varphi}']\bm{x} - cP \label{eq:objective_function}\\
  \text{subject to:} \quad
                    &M = \bm{x}'\bm{\mu}\label{eq:mean}\\
                    &V = \bm{x}'\bm{\sigma^2} \label{eq:variance}\\
                    &S = \sqrt{V}\label{eq:std}\\
                    &P \geq S\sum_{k=i}^W \dot{p}_k\dot{\mbox{E}}[\omega|\Omega_k]-(C-M)\sum_{k=i}^W\dot{p}_k +\dot{e}_WS & \quad i & = 1,\ldots,W\label{eq:pw_ub}\\
                    &P \geq e_WS \label{eq:pw_error}\\
                    &\bm{x} \in \{0,1\}^n. \label{eq:domain}\\
                    \nonumber
\end{align}
\end{minipage}
}
}

While piecewise linearisation has been leveraged before to embed the first order loss function in MILP models, the modelling strategy here introduced is different and, to the best of our knowledge, novel. For instance, in stochastic lot sizing \cite[see e.g.][]{Tarim2006}, one typically deals with a polynomial number of loss functions, this makes it possible to precompute the $\sigma$ terms of the expression in Lemma \ref{eq:standardized_piecewise_lb}, and to use binary variables to select the relevant loss function in a given constraint by relying on linear expressions. This is not possible in the case of the SSKP, which features an exponential number of loss functions that need to be considered --- one for each possible knapsack. Deviating from established modelling practices in stochastic inventory control, we therefore model the $\sigma$ term of Lemma \ref{eq:standardized_piecewise_lb} as a decision variable; this modelling choice leads to a non-linear and non-convex formulation stemming from the use of $\sqrt{}$ in \eqref{eq:std}. To revert to a fully linear formulation, we piecewise linearise $\sqrt{}$ (see \ref{sec:pw_sqrt}) by using the IBM ILOG OPL \texttt{piecewise} command \citep[][p. 150]{Nickel2022}. However, while linearising $\sqrt{}$, one must carefully select the number of segments, and build suitable piecewise linear upper or lower bounds. The model embeds Edmundson-Madanski's upper bound in \eqref{eq:pw_ub}. This means the model overestimates the first order loss function in the objective function; and in turn this means the model {\em underestimates} the true knapsack value. To preserve this behaviour, the model must embed a piecewise linear upper bound for $\sqrt{}$, so that $S$ --- the knapsack weight standard deviation --- is also overestimated.\footnote{Note that a reformulation $S^2 \geq V$ is not practicable as it is non-convex: it describes the region below the parabola $V = S^2$.} Conversely, to obtain the model under Jensen's lower bound, which {\em overestimates} the true knapsack value, it is sufficient to set $e_W=0$ in \eqref{eq:pw_ub}, and embed a piecewise linear lower bound for $\sqrt{}$, so that $S$ is underestimated.\footnote{In this case, reformulation $S^2 \leq V$ is convex, but under Jensen’s lower bound, the optimiser would minimise $S$ to zero to improve the objective, effectively collapsing the model to an expected-value formulation that ignores variability. Therefore piecewise linearisation of $\sqrt{}$ and an equality constraint \eqref{eq:std} are both necessary to preserve the intended role of $S$ as standard deviation.}
If both models return the same optimal knapsack, then this knapsack is optimal for the original problem; 
if not, we use Monte Carlo simulation to estimate the value of both these knapsacks and we retain the best one $x^*$. 
Finally, let $U$ be the value of the optimal solution of Jensen's lower bound model, and $L$ be that of Edmundson-Madanski's upper bound; we compute a conservative estimate of $x^*$ optimality gap as $(U-L)/L$.

\subsubsection{Multivariate normal weights}

\modified{Let us assume that normally distributed random weights in $\bm{\zeta}$ are correlated with variance-covariance matrix $\bm{\Sigma}$. The model presented in the previous section can be modified to accommodate this setup by replacing constraint \eqref{eq:variance} with the following quadratic inequality 
\begin{equation}\label{eq:variance_quadratic}
V \geq \bm{x}'\bm{\Sigma} \bm{x}, 
\end{equation}
where $\bm{x}'\bm{\Sigma} \bm{x}$ is convex ($\bm{\Sigma} \succeq 0$); moreover, by leveraging the law of total expectation \citep{Weiss2005}, it is easy to show that the objective function remains unchanged. In other words, the new model does not need to consider weight realisations, and can be built solely on the basis of unconditional means and covariances. In turn, this means that the SKP under multivariate normal weights can be tackled in a way that is similar to the case in which weights are independently distributed. The resulting mixed-integer quadratically constrained program is \eqref{eq:objective_function}, \eqref{eq:mean}, \eqref{eq:variance_quadratic}, \eqref{eq:std}, \eqref{eq:pw_ub}, \eqref{eq:pw_error}, and \eqref{eq:domain}. This model features convexity and semi-definiteness and can be solved by standard mathematical programming packages (e.g. CPLEX). Furthermore, since elements of $x$ are binary, the bilinear terms can be linearised by using McCormick envelopes \citep{McCormick1976}, but this typically adds a quadratic number of additional binary variables and it is not a practicable modelling path for large instances. Instead, we introduce reformulations to enhance model efficiency. First, let $\bm{I}_{nn}$ denote the $n \times n$ identity matrix and $ \mathbf{1}_{nn}$ the all-ones matrix. We exploit symmetry of the variance-covariance matrix and reformulate the quadratic inequality as follows: $V\geq \bm{x}' (\bm{\Sigma} \circ \bm{I}_n) \bm{x} + \bm{x}' (\bm{\Sigma} \circ (\mathbf{1}_n - \bm{I}_n)) \bm{x}$. This decomposition halves the number of quadratic terms in the constraint by separating diagonal and off-diagonal contributions. Second, we impose the bound $V\leq \mathbf{1}'_{n} \bm{\Sigma} \mathbf{1}_{n}$, which represents the maximum possible variance. Finally, we filter entries in $\bm{\Sigma}$ that are zero or negligible (i.e. below $10^{-6}$, the absolute tolerance of CPLEX).}

\subsubsection{Approximation accuracy}\label{subsec:accuracy}

Due to the interplay between the two piecewise linearisation strategies adopted to deal with the first order loss function and the $\sqrt{}$ function, existing results from the literature cannot be applied to analytically estimate the approximation accuracy of our mathematical programming approach.
In this section, we derive explicit conditions on the numbers of segments \(W+1\) (first-order loss) and \(Q\) (square-root) that guarantee a given absolute
optimality tolerance~\(\varepsilon>0\) for our piecewise linearised model. We are not aware of any similar analysis in the literature.

Let $V_{\max}$ be an upper bound for the variance of a knapsack weight.
For any \(Q\in\mathbb{N}\) let  
\(s\triangleq V_{\max}/Q\) and define break-points
\(b_k\triangleq ks\;(k=0,\dots,Q)\).
On each interval \([b_{k-1},b_k]\) we consider the chord
\[
   \psi_{Q}(v)\;=\;
     \sqrt{b_{k-1}}
     +\frac{\sqrt{b_k}-\sqrt{b_{k-1}}}{b_k-b_{k-1}}\,
      (v-b_{k-1}),
     \qquad v\in[b_{k-1},b_k].
\]
Because the square-root is concave, \(\psi_Q\) is a lower
piecewise-linear bound of \(\sqrt v\).
Using the standard construction of \ref{sec:pw_sqrt}, shifting it upward by
\(
  \Delta\triangleq \sqrt{s}/4
\)
yields an upper bound; we summarise both bounds as
\[
   \Psi_Q(v)
     \;=\;
     \begin{cases}
       \psi_Q(v),        &\text{lower bound},\\
       \psi_Q(v)+\Delta, &\text{upper bound}.
     \end{cases}
\]
The resulting worst-case deviation is
\begin{equation}\label{eq:deltaQdef}
   \delta_Q
      \;\triangleq\;
      \max_{v\in[0,V_{\max}]}
         \bigl|\sqrt{v}-\Psi_Q(v)\bigr|
      \;=\;\frac{\sqrt{s}}{4}.
\end{equation}

Recall that \(\dot e_{W}\) denotes the maximum approximation error of a $(W+1)$-segment Jensen bound of the standard normal loss.
After scaling to a normal random variable with variance~\(V\) via Lemma \ref{lemma:standardization} this yields
\vspace{-0.5em}
\begin{equation}\label{eq:lossScale}
   e_{W}\;=\;\dot e_{W}\,S,
   \vspace{-0.5em}
\end{equation}
where \(S\triangleq\sqrt V\) is the standard deviation of the
aggregate knapsack weight.

Fix a partition with \(W\) subregions and denote standard normal loss linearisation probability masses
and conditional means by \(\dot p_k\) and
\(\dot{\mbox{E}}[Z\mid\Omega_k]\) $(k=1,\dots,W)$.
For each
\(i=1,\dots,W\) define
\begin{equation}\label{eq:Aidef}
   A_i
   \;=\;
   \sum_{k=i}^{W}
     \dot p_k\,\dot{\mbox{E}}[Z\mid\Omega_k],
   \qquad
   A_{\max}
   \;=\;
   \max_{i=1,\dots,W}A_i.
\end{equation}

\begin{lemma}\label{lemma:Amax}
If $W$ is even, then $A_{\max}=\bar{A}_{\max} = \varphi(0) = 1/\sqrt{2\pi}$;
if $W$ is odd then $A_{\max}<\bar{A}_{\max}$. 
\end{lemma}
\begin{proof}
see \ref{sec:lemma_Amax_proof}.
\end{proof}
 
Let \(x\) be an arbitrary feasible 0-1 decision vector and let
\(Z^{\star}\) denote the true objective value.
Replacing the loss by the Edmundson--Madanski upper bound
increases the objective by at most
\(c\,e_{W}=c\,\dot e_{W}S\).
Afterwards, every occurrence of \(S\) in the objective \emph{and} in
the right-hand side of constraint~\eqref{eq:pw_ub}
is over-estimated by no more than~\(\delta_Q\).
Consequently the additional error introduced by the
square-root approximation does not exceed
\(c\,(A_{\max}+\dot e_{W})\,\delta_Q\).
Hence, define \(S_{\max}\triangleq\sqrt{V_{\max}}\), the absolute objective error satisfies
\begin{equation}\label{eq:totalErr}
   \bigl|Z(x)-Z^{\star}\bigr|
   \;\le\;
   c\,\dot e_{W}S_{\max}
   \;+\;
   c\,(A_{\max}+\dot e_{W})\,\delta_Q.
\end{equation}

\begin{lemma}[Global $\varepsilon$-accuracy]\label{lem:epsilon-new}
Let \(c>0\) and \(\varepsilon>0\) be given.
Choose integers \(W,Q\) such that
\begin{align}
  \dot e_{W} &\;\le\;\frac{\varepsilon}{2c\,S_{\max}},  \label{eq:condWnew}\\[2mm]
  \delta_Q   &\;\le\;\frac{\varepsilon}{2c\,(A_{\max}+\dot e_{W})},
                 \label{eq:condQnew}
\end{align}
where \(A_{\max}\) is defined in~\eqref{eq:Aidef} and
\(\delta_Q\) in~\eqref{eq:deltaQdef}.
Then both MILP relaxations satisfy
\begin{equation}\label{eq:gapStatement}
   0 \;\le\; Z^{\mathrm U}(W,Q)-Z^{\star}
      \;\le\; \varepsilon,
   \qquad
   0 \;\le\; Z^{\star}-Z^{\mathrm L}(W,Q)
      \;\le\; \varepsilon,
\end{equation}
and hence
\(Z^{\mathrm U}(W,Q)-Z^{\mathrm L}(W,Q)\le 2\varepsilon\).
\end{lemma}
\begin{proof}
Under~\eqref{eq:condWnew} the first term on the right-hand side of
\eqref{eq:totalErr} does not exceed \(\varepsilon/2\).
Under~\eqref{eq:condQnew} the second term is also bounded by
\(\varepsilon/2\).
Applying~\eqref{eq:totalErr} to the optimal solutions of the two MILPs
yields~\eqref{eq:gapStatement}.
\end{proof}

We finally characterise the asymptotic growth of \(W\) and \(Q\).
For a uniform Jensen partition one has
\(\dot e_{W}=O(W^{-2})\) --- see \ref{sec:convergence_pwl}; moreover,
\(\delta_Q=\Theta(Q^{-1/2})\) by~\eqref{eq:deltaQdef}.
Condition~\eqref{eq:condWnew} therefore implies
\(W=O(\varepsilon^{-1/2})\).
Substituting $\dot e_{W}=O(\varepsilon)$ and
$A_{\max}=O(1)$ into~\eqref{eq:condQnew} shows that
$Q=\Theta(\varepsilon^{-2})$,
so the required number of $\sqrt{\cdot}$ segments grows quadratically
as the tolerance $\varepsilon$ decreases. A numerical example illustrating these results is provided in \ref{sec:numerical_example}.

\section{The Dynamic Stochastic Knapsack}\label{sec:dskp}

\modified{In this section, we generalise the stochastic knapsack problem to a dynamic setting. Instead of formulating the problem as a two-stage stochastic program in which items must be selected at the beginning of the planning horizon, we assume that items arrive sequentially --- the item number denoting its position in the sequence; when item $i$ arrives, we must decide if it should be included in the knapsack or not. We assume that an item weight is revealed only after this decision is made, but the models and the approaches here presented can be easily modified to accommodate a situation in which the item weight is revealed after the item arrives, but before this decision is made. 
In contrast to most existing literature, we assume that item weights are potentially correlated. Thus, $\omega_i$ follows the conditional probability density function $g_{\omega_{i}}(o_i|\bm{O}_{i-1})$, where $\bm{O}_i\triangleq(o_1,\ldots,o_i)$, $\bm{O}_0\triangleq\emptyset$, and $o_i$ denotes the realisation of random variable $\omega_{i}$.

We formulate this problem as a stochastic dynamic program \citep{Bellman:1957}.
\paragraph{State} The state is a tuple $\langle i, C, \bm{O}_{i-1} \rangle$, where $i$ is the index of the item to be considered, $C$ is the residual capacity, and $\bm{O}_{i-1}$ are realised weights of items previously considered.
\paragraph{Action} The action is a binary variable $x_i\in\{0,1\}$ denoting item $i$ selection.
\noindent
\paragraph{Immediate profit} The immediate profit after selection of object $i$ is 
\[f_i(x, C, o_i)\triangleq\left\{
\begin{array}{lr}
x\mbox{E}[\varphi_i] & i<n\\
x\mbox{E}[\varphi_i] - c\max(x o_i - C) & i=n
\end{array}
\right.,\]
where $\varphi_i$ is the $i$th element of vector $\varphi$ as defined in Section \ref{sec:sskp}.
\paragraph{Functional equation} Let $F_i(C, \bm{O}_{i-1})$ denote the expected total profit of an optimal policy for the selection of items $i,\ldots,n$, given residual capacity $C$, and realisations $\bm{O}_{i-1}$ for previous items. Then, for $i=1,\ldots,n$
\[F_i(C, \bm{O}_{i-1})\triangleq \max_{x\in\{0,1\}} \int_{\omega_i} \Big(f_i(x, C, o) + F_{i+1}(C-o,\bm{O}_i)\Big)g_{\omega_i}(o|\bm{O}_{i-1}),\]
where $F_{n+1}\triangleq 0$ represents the boundary condition.}

\modified{\subsection{A receding horizon heuristic for the DSKP}\label{sec:receding_horizon}}

\modified{To tackle the stochastic knapsack problem in a dynamic setting (DSKP), we leverage the MILP models discussed for the SSKP under a receding horizon control framework \citep{Kwon2005-fs}. In contrast to rolling horizon control, in which a fixed-length time window rolls many times over a potentially infinite planning horizon, receding horizon control sets the length of the time window equal to the number of remaining periods in the planning horizon. In the context of the DSKP, receding horizon control proceeds as follows: when an item arrives, a selection plan is obtained for all remaining items, but only the imminent decision is implemented; re-planning is then carried out when the next item arrives, until the end of the item sequence. We will discuss the algorithm for the more general case in which item weights are correlated.

We first summarise some basic results from multivariate analysis. Let $\bm{\zeta}$ be an $n$-variate normal random variable with mean $\mbox{E}\bm{\zeta}$ and covariance matrix $\bm{\Sigma}$. Let $\bm{\zeta}\triangleq(\bm{\zeta}_1~\bm{\zeta}_2)'$ be a partitioned multivariate normal random $n$-vector, where $\bm{\zeta}_1\triangleq(\omega_1,\ldots,\omega_p)$, $\bm{\zeta}_2\triangleq(\omega_{p+1},\ldots,\omega_n)$, with mean $\mbox{E}\bm{\zeta}\triangleq(\mbox{E}\bm{\zeta}_1~\mbox{E}\bm{\zeta}_2)'$ and covariance 
\[\Sigma\triangleq
\left[
\begin{matrix}
\bm{\Sigma}_{11}&\bm{\Sigma}_{12}\\
\bm{\Sigma}_{21}&\bm{\Sigma}_{22}
\end{matrix}
\right].\] Then the conditional distribution of $\bm{\zeta}_2$ given $\bm{\zeta}_1=\bm{z}_1$ is normally distributed with mean $\mbox{E}\bm{\zeta}_2+\bm{\Sigma}_{21}\bm{\Sigma}_{11}^{-1}(\bm{z}_1-\mbox{E}\bm{\zeta}_1)$, and variance $\bm{\Sigma}_{22}-\bm{\Sigma}_{21}\bm{\Sigma}_{11}^{-1}\bm{\Sigma}_{12}$ \citep{Billingsley2017}. 

Our algorithm (see Algorithm \ref{alg:receding_horizon}) proceeds as follows: when item $k$ arrives, weight realisations for all previous items are available, we therefore update weight distributions for item $k$ and all successive items (line \ref{alg:update_conditional_distributions}) as discussed in the previous paragraph; then selection decisions for items $k,k+1,\ldots,n$ are obtained by solving one of the MILP models for the SSKP discussed in Section \ref{sec:milp_sskp} (line \ref{alg:solve_MILP}); the relevant decision for item $k$ is implemented (line \ref{alg:implement_decision}), item $k$ weight is observed, and knapsack profit (line \ref{alg:update_profit}) and residual capacity (line \ref{alg:update_capacity}) are updated; after the last item decision is implemented, shortage cost is incurred (line \ref{alg:shortage_cost}).}

\begin{algorithm}
\begin{algorithmic}[1]
\State knapsack value $v \gets 0$;
\State residual capacity $q \gets C$;
\For{$k \gets 1$ to $n$} 
    \State update the conditional distribution of $(\omega_{k},\ldots,\omega_n)$ given realisations $\bm{O}_{k-1}$; \label{alg:update_conditional_distributions}
    \State solve the MILP for a SSKP over items $k,\ldots,n$ with residual capacity $q$; \label{alg:solve_MILP}
    \State implement imminent decision $x_k$ for item $k$; \label{alg:implement_decision}
    \State $v \gets v + \varphi_k x_k$; \label{alg:update_profit}
    \State $q \gets q - o_kx_k$; \label{alg:update_capacity}
\EndFor
\State $v \gets v + c\min(q,0)$; \label{alg:shortage_cost}
\end{algorithmic}
\caption{Receding horizon heuristic}
\label{alg:receding_horizon}
\end{algorithm}

The heuristic discussed here can be immediately applied to tackle stochastic knapsack problems with weight forecast update following the Martingale Model of Forecast Evolution (MMFE) \citep{HEATH1994}. The AR, MA, and ARMA processes are special cases of the MMFE, therefore our approach can tackle all these time-series processes.

\modified{\subsection{A statistical optimality gap certificate for the receding horizon heuristic}
\label{sec:rh_certificate}

The receding horizon (RH) heuristic of Section \ref{sec:receding_horizon} yields a feasible policy for the DSKP. As such, its expected value is a lower bound on the optimal value of the stochastic dynamic program in Section \ref{sec:dskp}. 
In this section, we describe a statistical optimality gap certificate that combines this observation with a path-wise upper bound based on perfect information \cite[see e.g.][]{Birge2011-oa}; the resulting simulation-based certification procedure is related in spirit to Monte Carlo bounding approaches \citep{Mak1999}.

Let $x_k\in\{0,1\}$ denote the imminent decision prescribed by the RH policy at stage $k$, given residual capacity $C$ and previously observed weights $\bm{O}_{k-1}$. For a realised sample path $\bm{O}_n^{(m)}\triangleq(o_1^{(m)},\ldots,o_n^{(m)})$, let $\ell^{(m)}$ be the realised value obtained by simulating the RH policy on that path, namely by sequentially applying $x_k$, updating the residual capacity after each realised weight observation, and finally charging the shortage penalty. Since this policy is feasible for the original DSKP, if $F_1(C,\bm{O}_0)$ denotes the optimal value from the initial state, then
$
\mbox{\em E}[\ell^{(m)}]\leq F_1(C,\bm{O}_0)
$.

To obtain an upper bound, for the same realised sample path $\bm{O}_n^{(m)}$ we consider the deterministic knapsack problem in which all item weights are assumed to be known in advance. In other words, we solve the responsive knapsack problem induced by the weight vector $\bm{O}_n^{(m)}$, and let $u^{(m)}$ denote its optimal objective value:
\[
u^{(m)}
\triangleq
\max_{x\in\{0,1\}^n}
\left\{
\mbox{\em E}[\bm{\varphi}]'\bm{x}
-
c\max\!\left((\bm{O}_n^{(m)})'x-C,0\right)
\right\}.
\]
This is precisely the perfect-information relaxation of the problem for sample path $m$, hence it is optimistic and therefore
$
F_1(C,\bm{O}_0)\leq \mbox{\em E}[u^{(m)}]
$.
Quantity $\mbox{\em E}[u^{(m)}]$ is the expected value with perfect information (EVwPI).

Now suppose that $N$ sample paths are generated. We define
$\bar{\ell}_N \triangleq \frac{1}{N}\sum_{m=1}^{N}\ell^{(m)}$,
and
$\bar{u}_N \triangleq \frac{1}{N}\sum_{m=1}^{N}u^{(m)}$.
To improve stability, the same set of sampled paths is used to compute both $\bar{\ell}_N$ and $\bar{u}_N$. 
Let $s_{\ell}$ and $s_u$ denote the sample standard deviations of
$\ell^{(1)},\ldots,\ell^{(N)}$ and $u^{(1)},\ldots,u^{(N)}$, respectively. We then compute one-sided Student-$t$ confidence bounds:
\[
L_{1-\alpha/2}
\triangleq
\bar{\ell}_N
-
t_{N-1,1-\alpha/2}\frac{s_{\ell}}{\sqrt{N}},
\qquad
U_{1-\alpha/2}
\triangleq
\bar{u}_N
+
t_{N-1,1-\alpha/2}\frac{s_u}{\sqrt{N}},
\]
where $t_{N-1,1-\alpha/2}$ is the $(1-\alpha/2)$ quantile of the Student-$t$ distribution with $N-1$ degrees of freedom. By Bonferroni's inequality, with confidence at least $1-\alpha$, $L_{1-\alpha/2}$ is a valid lower bound for the value of the feasible RH policy and $U_{1-\alpha/2}$ is a valid upper bound for the EVwPI relaxation.

Finally, we define the absolute 
($
\Delta_{\mathrm{abs}}
\triangleq
U_{1-\alpha/2}-L_{1-\alpha/2},
$) and relative 
($
\Delta_{\mathrm{rel}}
\triangleq
(U_{1-\alpha/2}-L_{1-\alpha/2})/L_{1-\alpha/2}
$) certification measures. The quantity $\Delta_{\mathrm{abs}}$ provides a conservative certification measure for the distance between a feasible adaptive policy and a perfect-information relaxation, while $\Delta_{\mathrm{rel}}$ reports the same information on a relative scale. This certificate is inexpensive to compute, statistically meaningful, and can be utilised with the RH implementation discussed in this paper to obtain certificates for larger instances that cannot be solved to optimality using stochastic dynamic programming. 
Moreover, the upper bound can be tightened by leveraging the batch-means approach discussed in \cite{Mak1999}.
}

\section{Computational study}\label{sec:experiments}

\modified{This section presents an extensive numerical study on the computational performance and optimality gap of the proposed mathematical programming approach. After introducing our test bed, we examine both the static (SSKP) and dynamic (DSKP) versions of the stochastic knapsack problem.}

Numerical experiments are conducted using IBM ILOG CPLEX Optimization Studio 22.11 with default settings (except where stated otherwise) and Java on an Intel(R) Xeon(R) E-2146G CPU @ 3.50GHz, 3504 Mhz, 6 Core(s), 12 Logical Processor(s) with 64GB of RAM. Monte Carlo simulations \citep{Hammersley1964} are carried out via Latin Hypercube Sampling (LHS) \citep{McKay79} using Common Random Numbers \citep{Conway63} and $10^5$ runs. Relative tolerance for all approaches is set to $10^{-4}$ (equivalent to CPLEX default), which corresponds to a percent optimality gap of $10^{-2}$. All methods are assigned a time limit of 10 minutes. Code and numerical results: \url{https://github.com/gwr3n/skp}.

\subsection{Test bed generation}

We follow the benchmark design from \cite{Pisinger2005}. Let $R=100$  represent the so-called {\em data range}, that is the maximum value of an item expected weight and expected revenue per unit; note that this upper bound is comparable to that considered in \citep{Merzifonluolu2011}. 
We consider the following {\em instance types}, where $\mathrm{Unif}[a,b]$ denotes a uniform distribution over $[a,b]$.

\noindent
\textbf{U} (Uncorrelated): $\mbox{E}[\omega_i] \sim \mathrm{Unif}[1,R]$, and $\mbox{E}[\varphi_i] \sim \mathrm{Unif}[1,R]$, independently.

\noindent
\textbf{WC} (Weakly Correlated): $\mbox{E}[\omega_i] \sim \mathrm{Unif}[1,R]$; $\mbox{E}[\varphi_i] \sim \mathrm{Unif}[U,U+2R/10]$, where $U \triangleq \max(\mbox{E}[\omega_i]-R/10,1)$.

\noindent
\textbf{SC} (Strongly Correlated): $\mbox{E}[\omega_i] \sim \mathrm{Unif}[1,R]$; $\mbox{E}[\varphi_i] \triangleq \mbox{E}[\omega_i] + R/10$.

\noindent
\textbf{ISC} (Inverse Strongly Correlated): $\mbox{E}[\varphi_i] \sim \mathrm{Unif}[1,R]$; $\mbox{E}[\omega_i] \triangleq \mbox{E}[\varphi_i] + R/10$.

\noindent
\textbf{ASC} (Almost Strongly Correlated): $\mbox{E}[\omega_i] \sim \mathrm{Unif}[1,R]$; $\mbox{E}[\varphi_i] \sim \mathrm{Unif}[\mbox{E}[\omega_i]+R/10-R/500,\mbox{E}[\omega_i]+R/10+R/500]$.

\noindent
\textbf{SS} (Subset Sum): $\mbox{E}[\omega_i] \sim \mathrm{Unif}[1,R]$; $\mbox{E}[\varphi_i] \triangleq \mbox{E}[\omega_i]$.

\noindent
\textbf{USW} (Uncorrelated with Similar Weights): $\mbox{E}[\omega_i] \sim \mathrm{Unif}[R,R+10]$; $\mbox{E}[\varphi_i] \sim \mathrm{Unif}[1,R]$.

\noindent
\textbf{PC} (Profit Ceiling): $\mbox{E}[\omega_i] \sim \mathrm{Unif}[1,R]$; $\mbox{E}[\varphi_i] = d\lceil \mbox{E}[\omega_i]/d \rceil$, with $d=3$; i.e. expected revenue per unit is a multiple of a given parameter $d$.

\noindent
\textbf{C} (Circle): $\mbox{E}[\omega_i] \sim \mathrm{Unif}[1,R]$; $\mbox{E}[\varphi_i] = d\sqrt{4R^2 - (\mbox{E}[\omega_i] - 2R)^2}$, with $d=2/3$; i.e. expected revenue per unit as a function of the expected weights form an arc of an ellipsis.

We consider different problem sizes $n$, and to ``smooth out'' variations due to the choice of capacity as described in \citep{Pisinger1999} for each combination of $R$, instance type, and problem size, we generate $H=10$ instances whose capacity is defined as follows,
\[C_h\triangleq \frac{h}{H+1}\sum_{i=1}^{n}\mbox{E}[\omega_i]\qquad\mbox{for}~h=1,\ldots,H.\]  
Let $\mu_i$ and $\sigma_i$ be the expected value and the standard deviation of $\omega_i$, respectively. For the problem under normally distributed weights, we consider $n\in\{25,50,100,200,500\}$ and two levels for the coefficient of variation $c_v\in\{0.1,0.2\}$, where $c_v\triangleq\sigma_i/\mu_i$. 
For the problem under multivariate normal weights, we consider $n\in\{25,50,100\}$ two levels for the coefficient of variation $c_v$ as discussed above, and two levels for the correlation coefficient $\rho\in\{0.75,0.95\}$. 
For the problem under generic distributions, we consider $n\in\{25,50\}$ gamma and lognormal distributed weights $\omega_i$ with parameters such that $c_v=\sigma_i/\mu_i$.  
Finally, the shortage cost is $c=10$. The test bed comprises a total of 900 instances under normal distribution, 1080 instances under multivariate normal distribution, and 360 instances under both gamma and lognormal distributions. Our benchmark design generalises the one adopted by \cite{Fortz2012}, which in turn is based on \citep{Martello1999}.




\subsection{Static Stochastic Knapsack Problem}

\modified{In this section, we report computational results for the MILP approaches that tackle the SSKP. For brevity, we label the MILP model from Section \ref{sec:normal} as PWL. We first introduce relevant baselines from the literature (Section \ref{sec:baselines}); then we proceed and illustrate computational results for normal (Section \ref{sec:results_PWL_normal}), multivariate normal (Section \ref{sec:results_PWL_multi_normal}), and generic (Section \ref{sec:results_PWL_generic}) item weight distributions.}

\subsubsection{Baselines}\label{sec:baselines}

We consider the following four baselines from the literature. 

\vspace{-0.5em}
\paragraph{SAA} The first baseline implements the Sample Average Approximation of \cite{Kleywegt2002} with large‑deviation (LD) guided sizing. Instead of fixing the scenario size via their crude global bound $ \hat N_{\text{LD}} $ (2.23) --- which under our relative tolerance regime is far too conservative \cite[see][p.\ 485]{Kleywegt2002} --- we adaptively choose (N) using their Proposition 2.2 under Assumption A (Section 2.2). The two-phase implementation we adopted is outlined in \ref{sec:saa_implementation}.

\vspace{-0.5em}
\paragraph{LC} The second baseline, is the LP/NLP algorithm discussed in \citep[][Section 4.1.1]{Fortz2012} that relies on lazy (global) cuts implemented in CPLEX. \modified{To implement this method it is necessary to disable the dual presolve and dynamic search in line with the discussion in \citep[][p. 213]{Fortz2012}.} While we left all other CPLEX settings to their default values, it should be noted that a lazy (global) cuts strategy prevents CPLEX from using multithreading and other default features. As we will discuss later, this is clearly a major limitation of LC compared to other approaches here considered. Finally, LC is exact and, if it does not time out, yields an optimal solution within numerical tolerance.

\vspace{-0.5em}
\paragraph{KKT} The third baseline is the KKT-based heuristic discussed in \citep[][Section 4.2]{Merzifonluolu2011}. 

\vspace{-0.5em}
\paragraph{RNG}The fourth baseline is the shortest path-based approach of \cite{Range2018}, which we aligned to our problem setting by not enforcing the chance-constraint in their model --- this is achieved by setting the maximum overfill to a sufficiently large value --- and by only considering the penalty function for violation of the capacity constraint.

\subsubsection{Performance of PWL under normally distributed item weights}\label{sec:results_PWL_normal}

Before comparing our PWL approach against established baselines in the literature, we investigate its computational and optimality gap\footnote{Recall that, as discussed in Section \ref{sec:normal}, by solving Jensen's and Edmundson-Madanski’s variants of PWL, it is possible to directly obtain a conservative estimate of the optimality gap for the optimal solution obtained.} performance. In particular, we investigate the sensitivity of runtimes and of optimality gaps with respect to the number of first-order loss segments utilised. 

Recall PWL leverages piecewise linearisation of two functions: the first-order loss and the $\sqrt{}$. By leveraging Lemma \ref{lemma:Amax} and Lemma \ref{lem:epsilon-new}, we fix the square root piecewise linearisation step size ($s=0.1$), so that the component $c\,(A_{\max}+\dot e_{W})\,\delta_Q$ of the absolute objective error in \eqref{eq:totalErr} is negligible in the context of our test bed. The effect of the remaining component of the absolute objective error, i.e. $c\,\dot e_{W}S_{\max}$, dominates as it depends on $S_{\max}$, and can be reduced by increasing $W$ and thus reducing $e_{W}$; this is the aspect we investigate next by varying the number of partitions $W\in\{25,50,100\}$ used in the first-order loss piecewise linearisation. For the 900 instances considered under {\em normally distributed item weights} computational times and optimality gaps for varying $W$ are shown in Figure~\ref{fig:pwl_comp_time_segs} and Figure~\ref{fig:pwl_rel_opt_gap}, respectively.
PWL scales remarkably well in the number of first-order loss segments utilised. Moreover, in the context of our test bed, $W=100$ is sufficient to ensure most instances are solved to optimality (up to the prescribed relative tolerance) within the given time limit. In what follows, we will therefore adopt this $W$ value.
\begin{figure}
\centering
\includegraphics[width=\textwidth, angle=0]{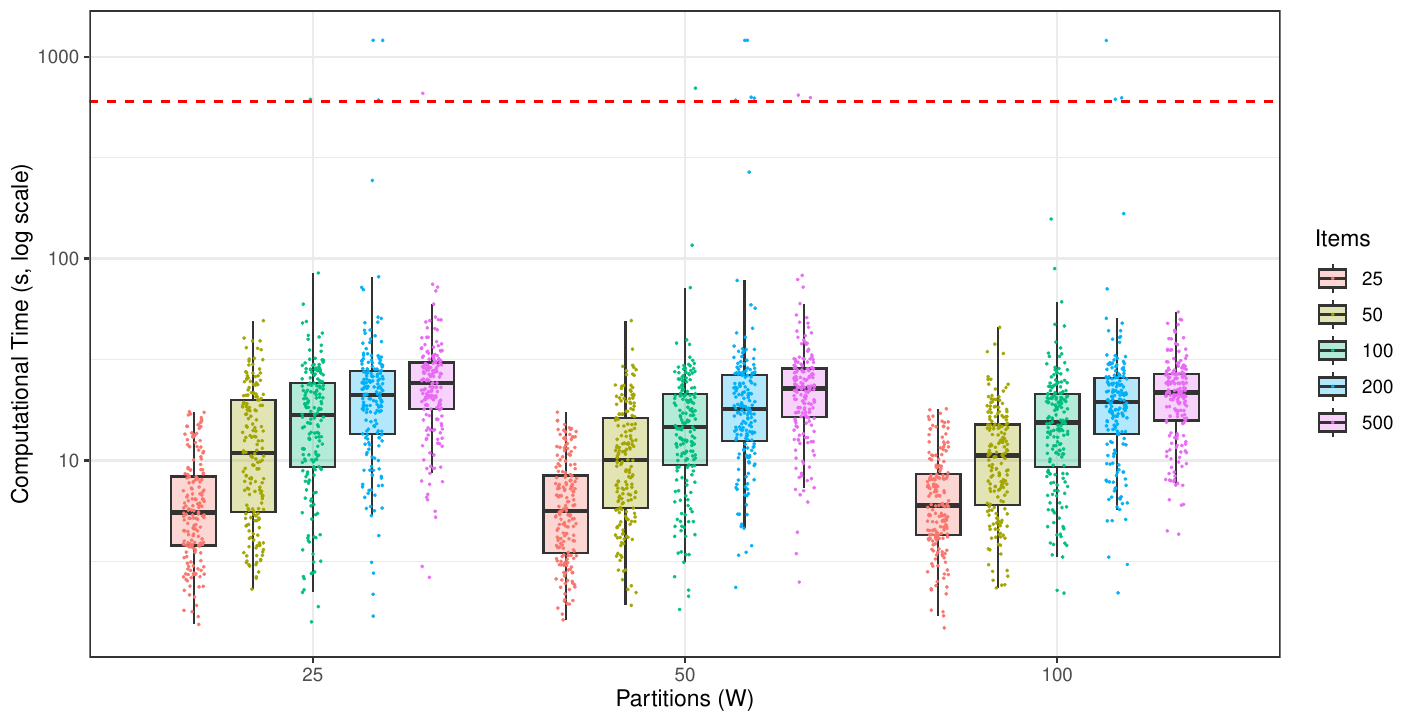}
\caption{Boxplot with jitter illustrating runtimes for PWL under normally distributed item weights and varying number of partitions ($W$); the dashed red line represents the time limit of 10 minutes.}
\label{fig:pwl_comp_time_segs}
\end{figure}
\begin{figure}
\centering
\includegraphics[width=\textwidth, angle=0]{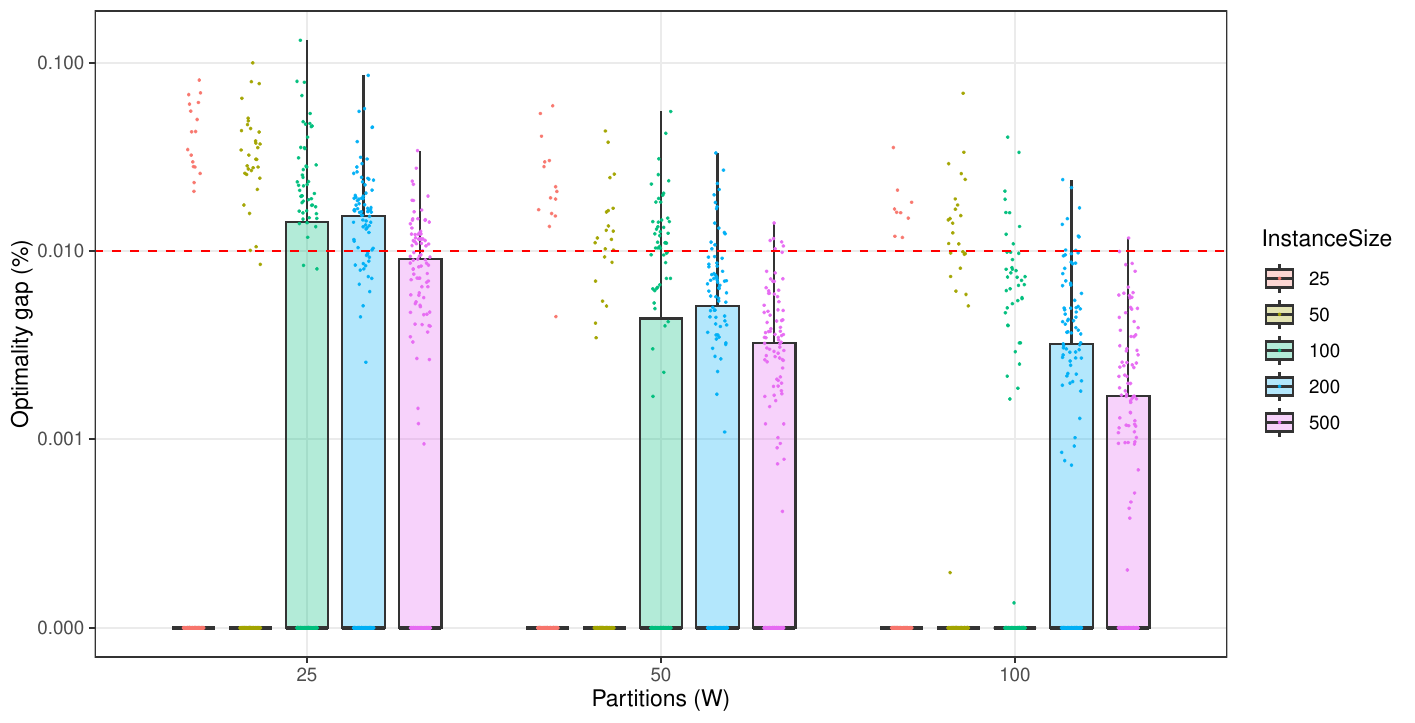}
\caption{Boxplot with jitter illustrating optimality gaps for PWL under normally distributed item weights and varying number of partitions ($W$); the dashed red line represents the target tolerance.}
\label{fig:pwl_rel_opt_gap}
\end{figure}

\newpage
PWL computational performance is stable across various instance types in the test bed (see tables in \ref{sec:pivot_tables_pwl}), with ISC instances being slightly more difficult to solve than others; average runtime increases from 7.04 seconds (25 items) to 22.0 seconds (500 items), and this increase is stable for varying $c_v$.
We observed 3 timeouts in total across the 900 instances. Over 95\% of the instances were solved to optimality (up to the prescribed relative tolerance); the maximum percent optimality gap observed was $6.8\cdot10^{-2}$ --- only slightly higher than the target percent relative tolerance $10^{-2}$.

\begin{figure}
\centering
\includegraphics[width=\textwidth, angle=0]{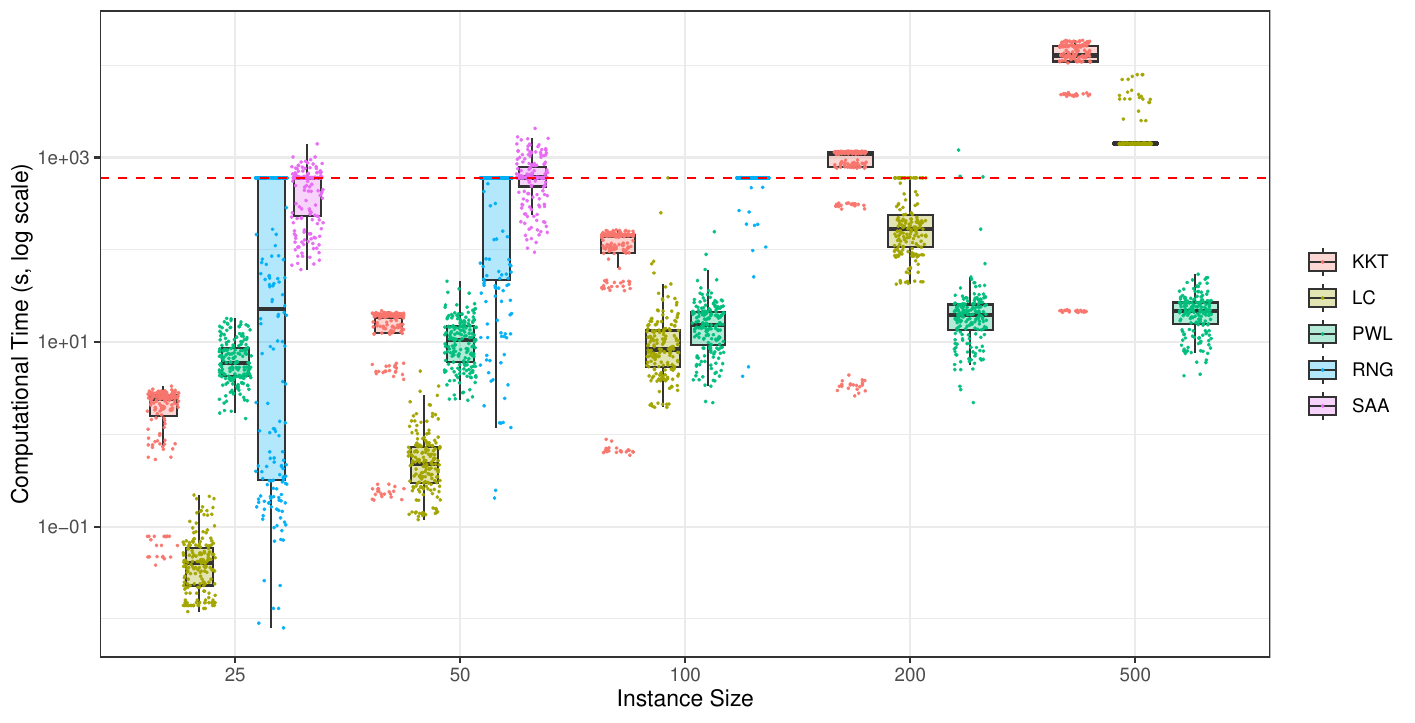}
\caption{Boxplot with jitter illustrating runtimes for PWL and the four baselines considered under normally distributed item weights; the dashed red line represents the time limit of 10 minutes; see \ref{sec:boxplots_KKT} for a discussion on KKT dot clustering.}
\label{fig:pwl_vs_baselines_normal}
\end{figure}
We  next compare PWL against the four baselines (PWL, SAA, LC and RNG) introduced in Section \ref{sec:baselines} for the 900 instances under normally distributed item weights.\footnote{Due to a known CPLEX issue, we could not strictly enforce the time limit. For large problems, CPLEX may exceed it since the limit applies only during optimisation---not during model construction, presolve, root relaxation, or heuristics. As a result, some solution times in our study surpass the set limit. While not officially documented, this behaviour is widely discussed on forums like \textit{IBM TechXchange}, \textit{Stack Overflow}, and \textit{OR Stack Exchange}. Given this limitation, we chose not to impose a time limit for KKT, allowing it to run to completion for all instances.} The comparison is illustrated in Figure~\ref{fig:pwl_vs_baselines_normal}.
Recall that our SAA baseline uses an LD‑driven scenario size: Phase~0 dynamically grows $ N $ until the LD bound is met, the time budget is exhausted, or $N$ reaches \(N_{\text{cap}}\). In our benchmark, even when considering only instances with 25 and 50 items, Phase~0 only successfully selected a suitable $N$ in 35\% of cases; when it did succeed, the median $ N $ selected was 1024; and when it failed, the median of the last $ N $ reached was 2048, while the respective median crude global bound $ \hat N_{\text{LD}} $ was $2\cdot10^7$ --- indeed a very conservative figure. This phase is followed by a confidence‑based replication loop whose aim is to certify optimality of a solution. In Phase~1, only 36.9\% of certified gaps crossed the target tolerance before hitting the wall clock or the maximum number of replications $ M_{\max} $ (wall clock hits: 225, $M_{\max}$ hits: 2), with the average percent gap for uncertified instances stabilising at $1.74\cdot 10^{-1}$, one order of magnitude above tolerance. Under our resource and tolerance regime, SAA with LD‑driven sizing is therefore non‑viable as a time‑to‑solution competitor.
Likewise, RNG perform poorly even for the smallest instances, and times out for most of the larger instances; this is the reason why in Figure~\ref{fig:pwl_vs_baselines_normal} we show SAA results only for instance sizes 25 and 50, and RNG results only for instance sizes 25, 50, 100. 
Both LC and KKT are very fast for small instances, but their computational time grows exponentially as the number of items increases. Conversely, PWL is slightly slower for small instances, but its computational time grows at a lower rate compared to other approaches with the number of items. For very large instances comprising 200 and 500 items, PWL is orders of magnitude faster \modified{in our computational study.}

Finally, we investigated the performance of a ``warm start'' approach in which a PWL model with a limited number of segments (i.e. $W=25$) is paired with LC of \cite{Fortz2012}. Results are presented in \ref{sec:pivot_tables_pwl} and show that the performance of this strategy does not substantially differ from LC. Our analysis suggests that the performance gap observed with respect to PWL is likely caused by the missing components of CPLEX (e.g. dynamic search, dual presolve) that are disabled to accommodate lazy-constraint callbacks; \modified{this is further confirmed by the fact that we observe a noticeble drop in performance if we disable these components for PWL (see Table \ref{tab:solution_time_sskp_PWL_ablation}).}

\newpage
\subsubsection{Performance of PWL under multivariate normal distributed item weights}\label{sec:results_PWL_multi_normal}

Once more, before comparing our PWL approach against established baselines in the literature, we illustrate its performance. In this section, we limit our study to instances under  {\em multivariate normal distributed item weights} comprising up to 100 items, as for larger item sizes all approaches time out for most instances.

An analysis similar to that carried out in the previous section (see \ref{sec:analysis_opt_gap_comp_time_mvn}) reveals that a linearisation setup $W=100$ and $s=0.1$ is also appropriate in this context. We solved 1080 instances (see tables in \ref{sec:pivot_tables_pwl}). PWL presents some degree of variability linked to instance types across the test bed: SC, ISC, ASC and USW instances being the hardest to solve; average runtime increases from 8.2 seconds (25 items) to 75.1 seconds (100 items), and this increase is stable for varying $c_v$ and $\rho$. We observed 11 timeouts. Around 94\% of the instances were solved to optimality (up to the prescribed relative tolerance). The average, median, and maximum percent optimality gap observed across the whole test bed were 0.001, 0.000, and 0.068, respectively. This demonstrates that PWL remains near optimal even when correlation among item weights is considered and modelled via a multivariate normal distribution. 

Next, we consider the performance of PWL against the following two baselines from the literature: SSA and the LC approach of \citep{Fortz2012}. The comparison is illustrated in Figure~\ref{fig:pwl_vs_baselines_multinormal}.

\begin{figure}
\centering
\includegraphics[width=0.95\textwidth, angle=0]{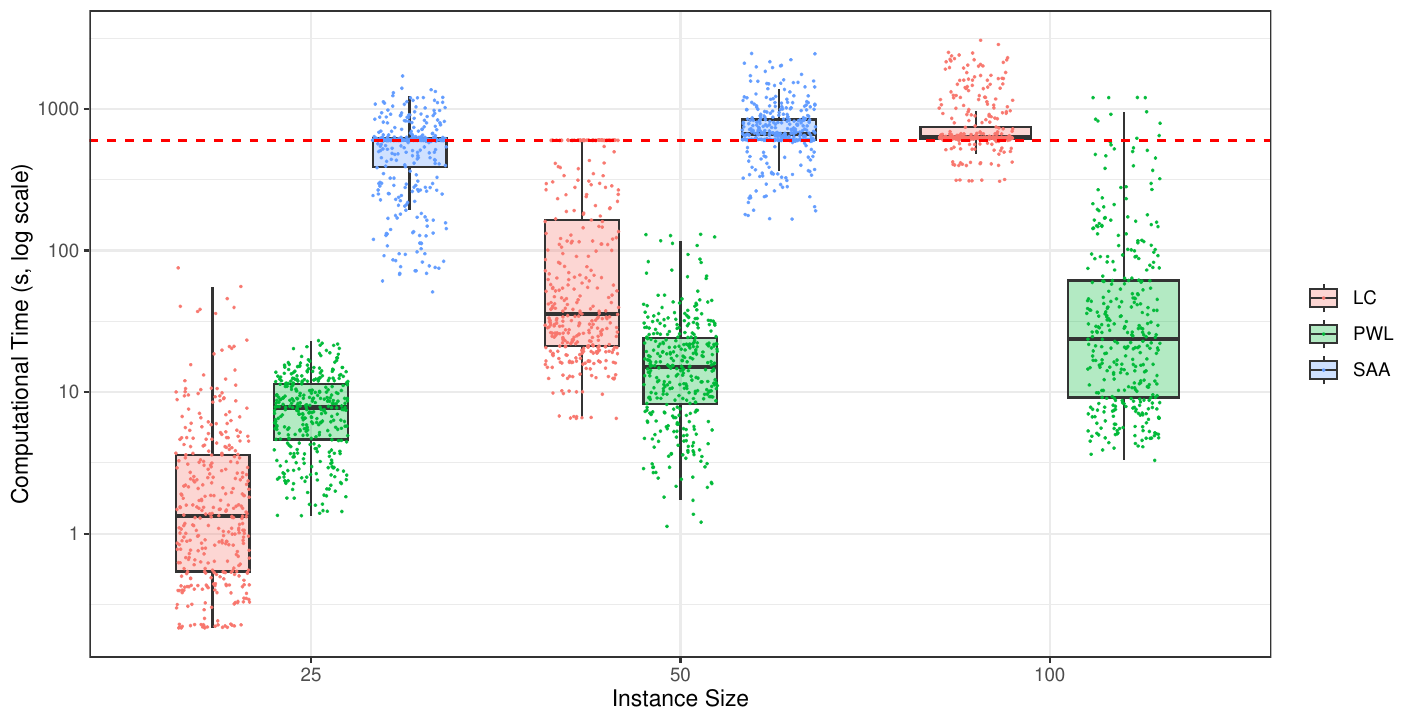}
\caption{Runtimes for PWL and the two baselines considered under multivariate normal distributed item weights; the dashed red line represents the time limit of 10 minutes.}
\label{fig:pwl_vs_baselines_multinormal}
\end{figure}

Once more, SAA performed poorly even for the smallest instances. LC is again very fast for small instances, but its computational time grows exponentially as the number of items increases. Conversely, PWL is slightly slower for small instances, but its computational time grows at a lower rate compared to other approaches as the number of items increases. For large instances comprising 100 items PWL scales better. To the best of our knowledge, these are the first computational results in the literature investigating the SSKP under multivariate normal demand.

\subsubsection{Performance of LC under item weights following a generic distribution}\label{sec:results_PWL_generic}

We consider the LC approach of \cite{Fortz2012}, which relies on lazy (global) cuts implemented in CPLEX. In their work, despite designing an algorithm that operates under generic distributions, the authors limited their computational study to the special case of normal weights, for which they generated --- as we did in the previous sections --- cuts by using efficient closed form expressions for the loss function \citep[see e.g.][Section C.3.2]{lawrencesnyder2011}. In this section, we generalised their approach and included suitable numerical integration strategies to generate cuts for instances comprising item weights that follow generic distributions, then we applied the approach to the 360 instances under gamma and lognormal distributed item weights ($n\in\{25,50\}$) from our test bed (see tables in \ref{sec:pivot_tables_lc}). The loss function is estimated via Monte Carlo simulation and Latin Hypercube Sampling (1000 replications) using Common Random Numbers. 
LC performance across the test bed varies noticeably depending on instance types, while it is stable for varying coefficients of variation; average runtime increases substantially with the size of an instance: from 262 seconds (25 items) to 450 seconds (50 items) for gamma item weights, and from 34.1 seconds (25 items) to 271 seconds (50 items) for lognormal item weights. Only 55.2\% (gamma) and 82.2\% (lognormal) of the instances could be solved to optimality.

As baseline, we consider SAA; the comparison is illustrated  in Figure~\ref{fig:pwl_vs_baselines_gamma} and Figure~\ref{fig:pwl_vs_baselines_lognormal}.
\begin{figure}[ht!]
\centering
\begin{minipage}{0.48\textwidth}
\centering
\includegraphics[width=\textwidth]{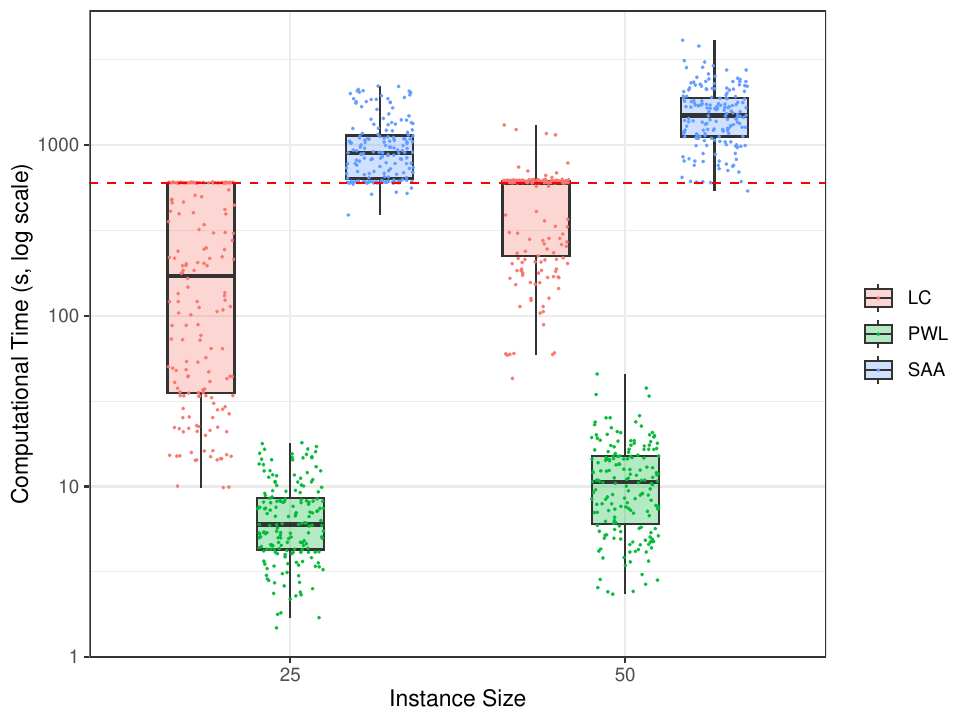}
\caption{Runtimes for LC and the SAA baseline considered under gamma distributed item weights; the dashed red line represents the time limit of 10 minutes.}
\label{fig:pwl_vs_baselines_gamma}
\end{minipage}
\hfill
\begin{minipage}{0.48\textwidth}
\centering
\includegraphics[width=\textwidth]{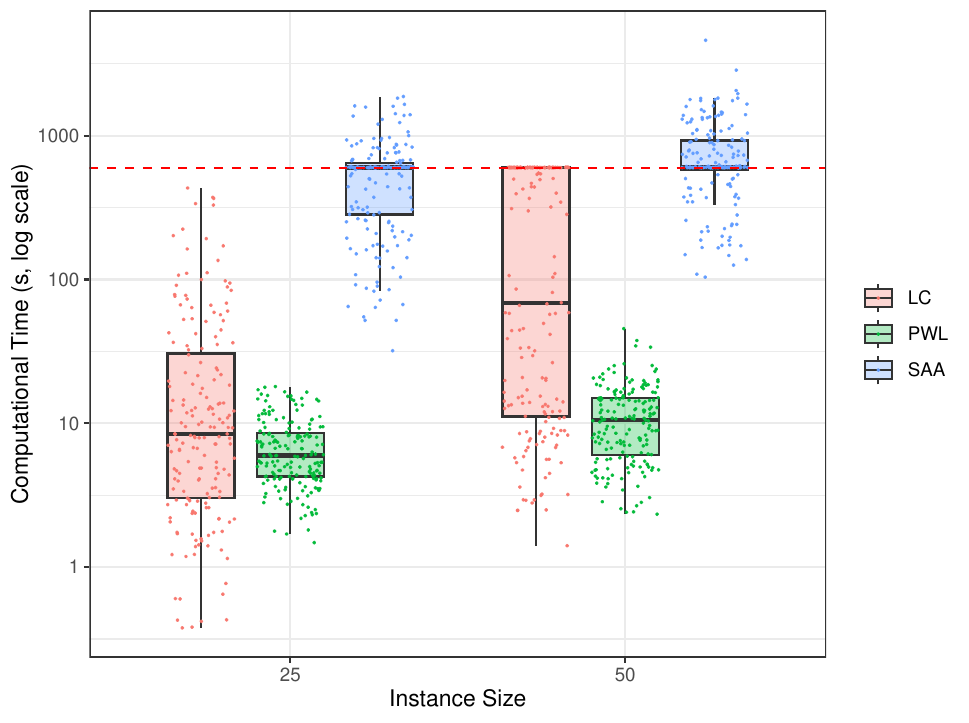}
\caption{Runtimes for LC and the SAA baseline considered under lognormal distributed item weights; the dashed red line represents the time limit of 10 minutes.}
\label{fig:pwl_vs_baselines_lognormal}
\end{minipage}
\end{figure}
SAA could only solve 2.5\% (gamma) and 38\% (lognormal) of the instances to optimality within the given time limit.

In the next section, since we limit our study to small 25-period instances, we adopt LC while investigating the performance of our receding horizon (RH) heuristic under item weights following a generic distribution, as this approach is exact and efficient enough for our purposes. However, computational results presented in this section suggest that a general purpose lazy constraint generation strategy paired with Monte Carlo simulation for computing the knapsack loss is not particularly scalable. 

\newpage
Unfortunately, if item weights follow generic distributions, it is not possible to rely on a closed form expression for computing the loss function, and one must therefore resort to numerical integration as suggested by \citep{Fortz2012}. In practice, however, if an instance comprises item weights following a generic distribution, as the instance size increases, the weight distribution of a knapsack will be approximately normal by the Central Limit Theorem (CLT) --- an argument that was similarly presented by \cite{Range2018} to motivate their heuristic. Because of this observation, a decision maker faced with such an instance may then decide to approximate item weights by their normal counterparts (e.g. via moment matching), and adopt our PWL approach, in lieu of LC or SAA, as an approximation strategy.

We carried out further experiments to investigate the effect of this approximation on the quality of the solutions obtained. More specifically, we solved instances via PWL by replacing the gamma (resp. lognormal) distribution with its normal counterpart featuring the same mean and standard deviation --- note that PWL computational times (reported in Figure~\ref{fig:pwl_vs_baselines_gamma} and Figure~\ref{fig:pwl_vs_baselines_lognormal}) remain those discussed in Section \ref{sec:results_PWL_normal} and are orders of magnitude smaller than those achieved by LC or SAA. Then we considered all instances in our test bed under gamma (199 instances) and lognormal (296 instances) distributions for which $n\in\{25,50\}$ and {\em for which LC did not time out}, and we compared the simulated value of the solution obtained by PWL against the value of the optimal solution obtained via LC --- recall that LC, when it does not time out, is exact up to chosen relative tolerance; its solution can be therefore used as reference to derive an optimality gap.
The average, median, maximum, and 95\% quantiles of the optimality gaps observed are tabulated in Table \ref{tab:optimality_gap_normal_approx_statistics}.
\begin{table}[ht!]
\centering
\begin{tabular}{l|rrrr}
 \toprule
 & Average & Median & Max & 95\% quantile\\
\midrule
gamma & 0.071 & 0.055 & 0.824 & 0.351\\
lognormal & 0.165 & 0.114 & 1.84 & 0.660\\
\bottomrule
\end{tabular}
\caption{Summary statistics for the average optimality gap (\%) of the PWL normal approximation heuristic to the SSKP under gamma and lognormal distributed item weights}
\label{tab:optimality_gap_normal_approx_statistics}
\end{table}

\noindent
The behaviour of the average optimality gap (\%) for different instance types is illustrated in Table \ref{tab:optimality_gap_normal_approx}. 
\begin{table}[ht!]
\centering
\begin{tabular}{l|rrrrrrrrr|rr}
 \toprule
 &\multicolumn{9}{|c|}{{\em instance types}}&\multicolumn{2}{|c}{$c_v$}\\
 & U & WC & SC & ISC & ASC & SS & USW & PC & C & 0.1 & 0.2 \\
\midrule
gamma & 0.05 & 0.00 & 0.11 & 0.31 & 0.02 & 0.38 & 0.08 & 0.26 & 0.04 & 0.06 & 0.09 \\
lognormal & 0.06 & 0.01 & 0.15 & 0.55 & 0.18 & 0.32 & 0.09 & 0.27 & 0.07 & 0.10 & 0.23 \\
\bottomrule
\end{tabular}
\caption{Average optimality gap (\%) of the PWL normal approximation heuristic to the SSKP under gamma and lognormal distributed item weights for different instance types}
\label{tab:optimality_gap_normal_approx}
\end{table}

While one may reasonably expect the CLT argument to apply only to instances comprising a substantial number of items, our results suggest instead that even for instances comprising a relatively small number of items ($n\in\{25,50\}$) approximating item weights by their normal counterparts represents an effective approximation strategy producing optimality gaps that generally do not exceed 0.5\%; and that our mathematical programming approach PWL therefore represents a computationally efficient heuristic for the SSKP with item weights following a generic distribution.

\subsection{Dynamic Stochastic Knapsack Problem}

In this section, we report computational results for the RH heuristics presented in Section \ref{sec:receding_horizon} to tackle the DSKP. This heuristic operates by iteratively solving $n$ SSKP problems of decreasing size, therefore its computational efficiency is linear in the efficiency of the model used to solve each of these problems. 

Since we have already extensively discussed the computational efficiency of our approaches to tackle the SSKP in the previous section, in this section we will concentrate on the optimality gaps of the RH heuristic. These optimality gaps are computed against the optimal dynamic policy obtained via stochastic dynamic programming. Since solving these instances to optimality via stochastic dynamic programming and estimating the performance of the RH heuristic are cumbersome tasks, we will limit our analysis to instances comprising 25 items. For the sake of comparison, we will also report optimality gaps of the static policy obtained via the SSKP model presented in Section \ref{sec:milp_sskp} when compared against the optimal dynamic policy. The performance of the RH heuristic and that of the static policy are estimated via Monte Carlo simulation.

For the 180 instances considered under normally distributed item weights, results are summarised via a boxplot with jitter in Figure~\ref{fig:optimality_gap_dskp_normal_boxplot}. The average optimality gap observed for the RH heuristics was 1.11\%. In comparison, the static policy obtained via the SSKP model presented in Section \ref{sec:normal} featured an average optimality gap of 3.80\%. 
\begin{figure}[ht!]
\centering
\begin{minipage}{0.48\textwidth}
\centering
\includegraphics[width=\textwidth, angle=0]{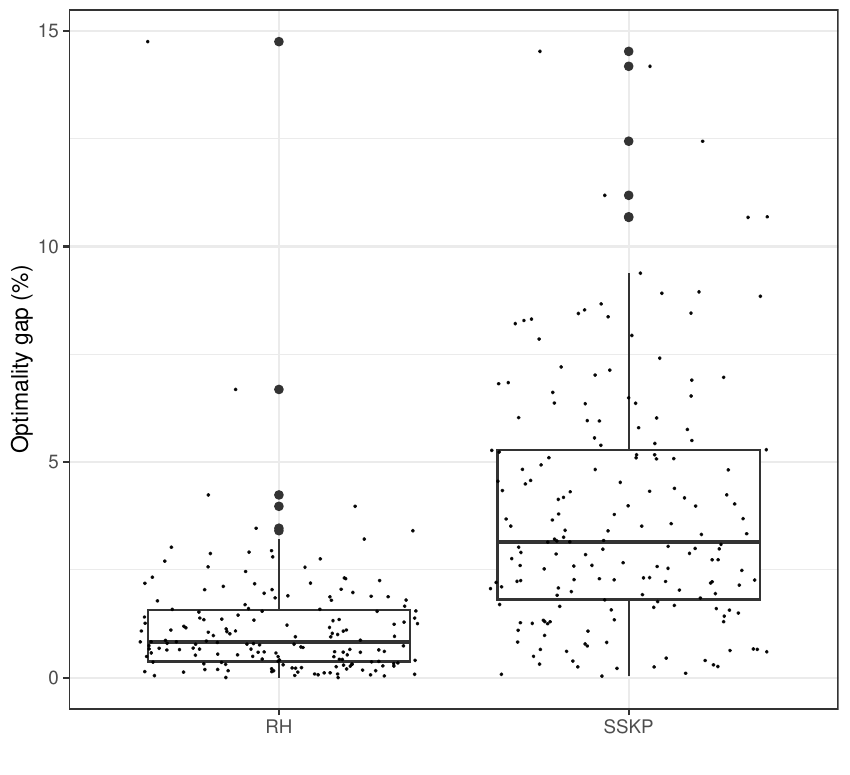}
\caption{Optimality gaps (\%) for the RH heuristics presented in Section \ref{sec:receding_horizon} and the static policy obtained via the SSKP model presented in Section \ref{sec:normal} under normally distributed item weights}
\label{fig:optimality_gap_dskp_normal_boxplot}
\end{minipage}
\hfill
\begin{minipage}{0.48\textwidth}
\centering
\includegraphics[width=\textwidth, angle=0]{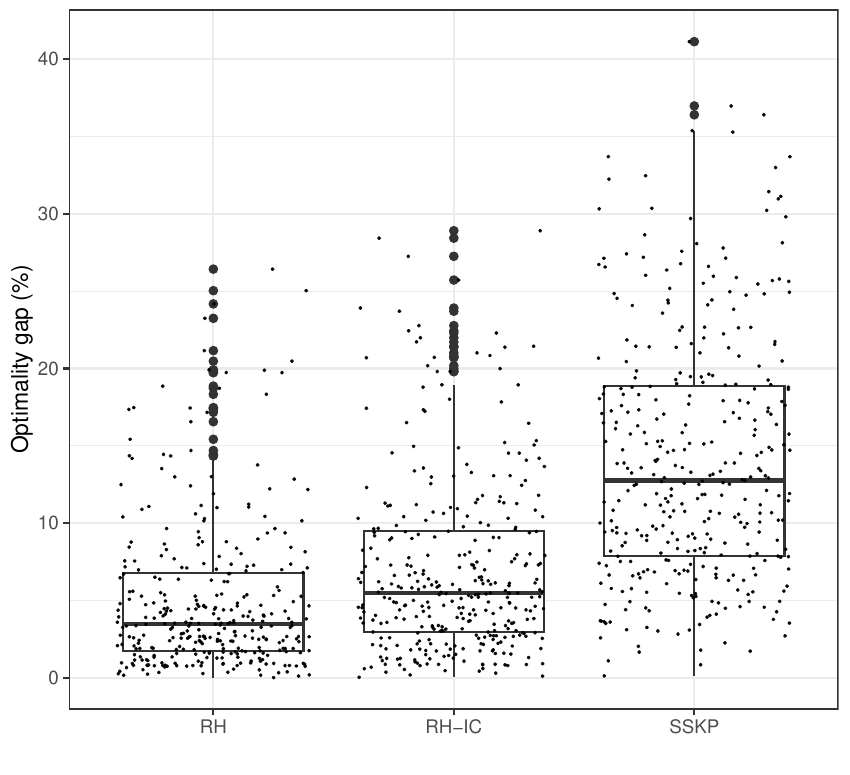}
\caption{Optimality gaps (\%) for the RH heuristic presented in Section \ref{sec:receding_horizon} and the static policy obtained via the SSKP model presented in Section \ref{sec:normal} under multivariate normal item weights; RH-IC denotes a receding
horizon heuristic in which the underpinning SSKP model ignores correlation among items.}
\label{fig:optimality_gap_dskp_multivariate_normal_boxplot}
\end{minipage}
\end{figure}

As discussed in Section \ref{sec:receding_horizon}, under a receding horizon framework, if weights follow a multivariate normal distribution, once the first $k$ item weights have been observed, we must update successive item weight distributions by using the conditional distribution. Since weights are correlated across multiple items, it is prohibitive to find the optimal solution using stochastic dynamic programming. We therefore impose a special structure $\sigma_{ij}^2=\rho^{|j-i|}\sigma_i\sigma_j$ to construct the covariance matrix. This ensures that the computation of the conditional mean only requires knowledge of the realised weight of the last item observed; in turn, this reduces the size of the state space that needs to be considered by stochastic dynamic programming. Nevertheless, even with this special structure in place, for the 360 instances considered under multivariate normal distributed item weights, it has been necessary to set the data range $R=10$ in order to generate instances that could be solved via stochastic dynamic programming in reasonable time. The average optimality gap observed for the RH heuristic was 4.49\%. In comparison, the static policy obtained via the SSKP model presented in Section \ref{sec:normal} featured an average optimality gap of 13.8\%. To illustrate the value of considering correlation among item weights while determining the optimal control policy, we also solved the test bed by using a receding horizon heuristic in which the underpinning SSKP model ignores correlation among items (RH-IC); for this heuristic the average optimality gap observed was 6.57\%. Results are summarised via a boxplot with jitter in Figure~\ref{fig:optimality_gap_dskp_multivariate_normal_boxplot}.
 
\begin{figure}[ht!]
\centering
\begin{minipage}{0.48\textwidth}
\centering
\includegraphics[width=\textwidth]{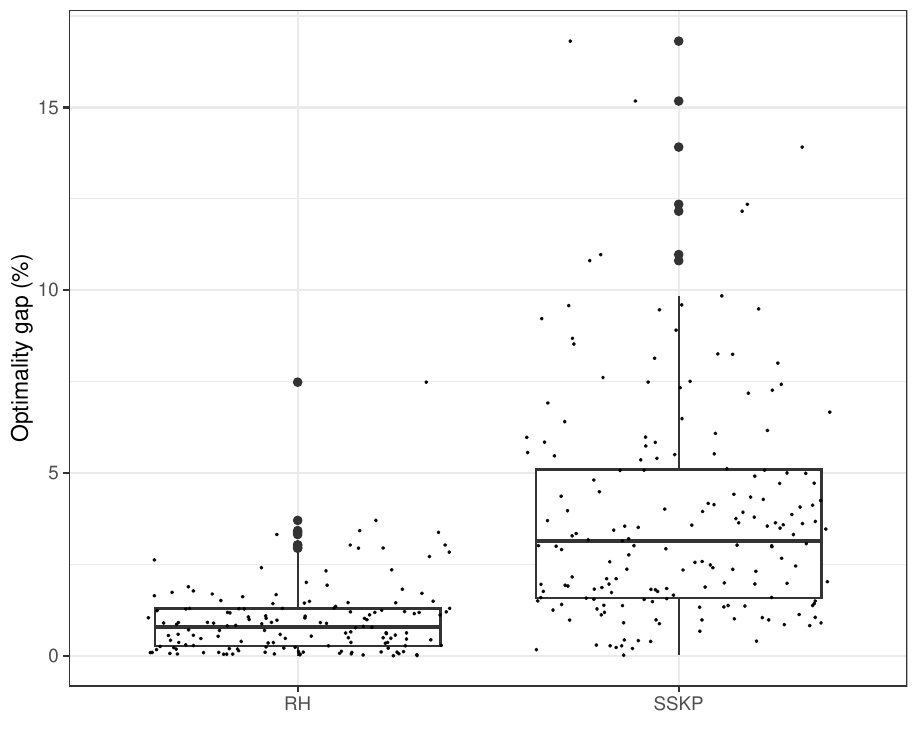}
\caption{Optimality gaps (\%) for the RH heuristics presented in Section~\ref{sec:receding_horizon} and the static policy obtained via the LC approach under gamma distributed item weights.}
\label{fig:optimality_gap_dskp_gamma_boxplot}
\end{minipage}
\hfill
\begin{minipage}{0.48\textwidth}
\centering
\includegraphics[width=\textwidth]{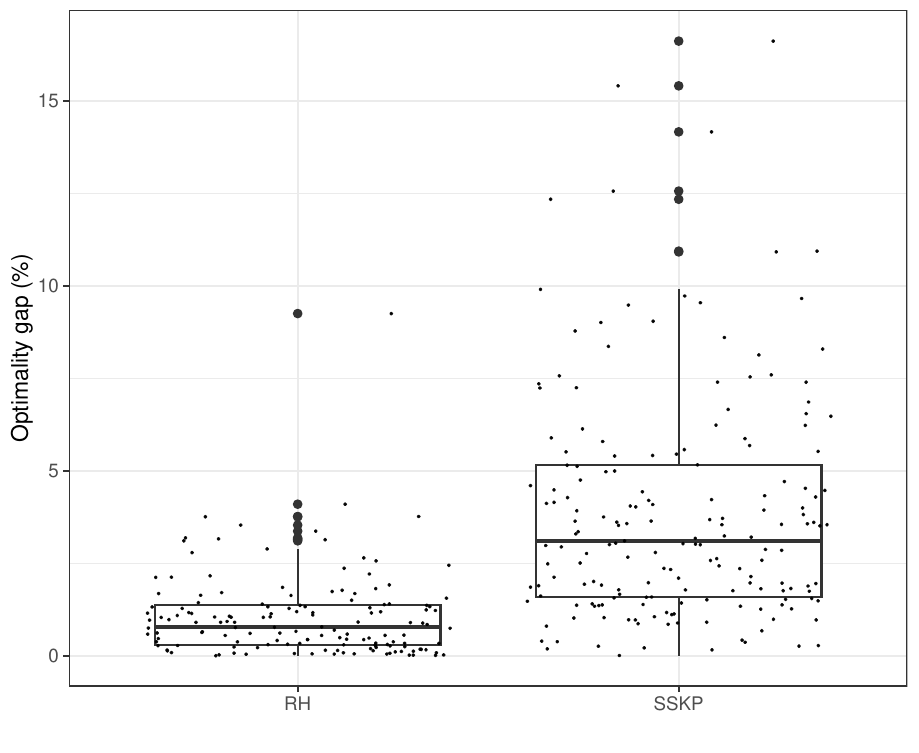}
\caption{Optimality gaps (\%) for the RH heuristics presented in Section~\ref{sec:receding_horizon} and the static policy obtained via the LC approach under lognormal distributed item weights.}
\label{fig:optimality_gap_dskp_lognormal_boxplot}
\end{minipage}
\end{figure}

For the 180 instances considered under gamma and lognormal distributed item weights, results are summarised via a boxplot with jitter in Figure~\ref{fig:optimality_gap_dskp_gamma_boxplot} and Figure~\ref{fig:optimality_gap_dskp_lognormal_boxplot}. The average optimality gap observed for the RH heuristics was 0.785\% (gamma) and 0.849\% (lognormal). In comparison, the static policy obtained via the LC approach of \cite{Fortz2012} featured an average optimality gap of 3.83\% (gamma) and 3.84\% (lognormal).

\modified{Finally, in \ref{sec:statistical_certificates_RH} we contrast the optimality gaps for the receding horizon heuristic computed against the optimal policy obtained via stochastic dynamic programming and those obtained by means of the statistical certificates produced by the approach in Section \ref{sec:rh_certificate} for instances under normally distributed item weights; as expected, this analysis reveals that these latter certificates are valid but typically conservative.}

\section{Conclusions}\label{sec:conclusions}

\modified{We presented mathematical programming reformulation and approximation strategies to tackle the SSKP under normally distributed item weights. In contrast to other competing approaches discussed in the literature, our approach extends to the case in which item weights are correlated and follow a multivariate normal distribution. 
We also investigated a receding horizon control framework to tackle the DSKP, in which items are considered sequentially, and a decision on an item inclusion may depend on previous item weight realisations. This latter framework can be used in conjunction with our mathematical programming approximations --- for (multivariate) normal item weights --- or with other existing approaches from the literature  \citep[e.g.][]{Kleywegt2002,Fortz2012} to tackle generic item weight distributions.}

In an extensive computational study, we compared our SSKP approach against four state-of-the-art baselines from the literature on a large benchmark comprising instances under (multivariate) normal item weights. 
We demonstrated that our solution method is near-optimal (up to a prescribed numerical tolerance) and that it is more scalable than these state-of-the-art baselines for very large instances. 

Motivated by the observation that, under generic item weight distributions, as an instance size increases the weight distribution of a knapsack is likely to be approximately normal by the Central Limit Theorem \citep{Range2018}, we also empirically demonstrated that approximating generic item weight distributions with normal ones yields near-optimal solutions with optimality gaps generally below $0.5\%$ in our benchmark study. In turn, this means our SSKP approach, despite being conceived for (multivariate) normal item weights, can be  deployed as a fast and effective heuristic under item weights following a generic distribution.

Finally, we investigated the optimality gap performance of the RH heuristics for the DSKP and found average optimality gaps of 1.1\% for normally distributed weights, 4.5\% for multivariate normal distributed weights, and around 0.8\% for both gamma and lognormal distributed weights. 
Under multivariate normal distributed weights the average optimality gap of RH heuristics increases to 6.6\% when correlation among item weights is ignored, this demonstrates the importance of considering item weight correlation.

While we have not considered the MMFE, or its underpinning time-series processes such as AR or MA, our models can be directly applied to these time-series processes, which are all built upon the multivariate normal distribution. Moreover, as in \citep{Merzifonluolu2011}, it is possible to extend our study and investigate variants of the problem in which capacity is random or can be purchased at a given cost per unit. 

Future research may further investigate the reasons underpinning low scalability of lazy cut generation-based approaches such as LC, and strategies for overcoming the performance gap created by the need to disable CPLEX features to enable lazy cut generation. \modified{Additionally, one may investigate the intrinsic strength of the formulations presented in this work by means of LP-relaxation analysis, integrality-gap comparisons, and polyhedral studies.} Finally, a recent study \citep{Jooken2022-yb} investigated structurally hard deterministic knapsack instances that take a long time to solve despite being small; an interesting direction of enquiry is to establish what makes a SSKP instance 
structurally hard, and to contrast these characteristics with those of its deterministic counterpart. 
\clearpage

\appendix

\section{Piecewise linearisation of the square root function}\label{sec:pw_sqrt}

Consider the square root function $\sqrt{x}$; within interval $[0,B]$, where $B$ is an arbitrary upper bound, we aim to build piecewise linear lower and upper bounds for this function.
We partition the interval $[0,B]$ by using a step size $s$ and we build one linear segment for each step. 

A piecewise linear lower bound is trivially constructed by taking breakpoints $\{0,s,2s,\ldots,B\}$ and slopes $\{1/\sqrt{s},\ldots,(\sqrt{b_i}-\sqrt{b_{i-1}})/s,\ldots\}$, where $b_i$ is the $i$th breakpoint. 

To construct a piecewise linear upper bound, we must shift up the piecewise linear lower bound by an amount $x_0$, so that one of the segments become tangent to the square root curve, and all other segments are above it. To determine $x_0$, recall that 
\[
\frac{d}{dx} \sqrt{x} = \frac{1}{2\sqrt{x}},
\]
and observe that this derivative monotonically decreases as $x\rightarrow \infty$; $x_0$ is therefore equal to the maximum distance between the first linear segment and the original function. To determine $x_0$, we solve for $x$
\[\frac{d}{dx} \left(\sqrt{x}-x\frac{\sqrt{s}-\sqrt{0}}{s}\right) = 0\]
and obtain $x_0=\sqrt{s}/4$, which is the amount one needs to shift up the piecewise linear lower bound to obtain the piecewise linear upper bound (Fig. \ref{fig:ub_lb_sqrt}).
\begin{figure}
\centering
\resizebox{0.3\textwidth}{!}{
\begin{tikzpicture}
    \begin{axis}[
        domain=0:10,
        samples=100,
        xlabel={$x$},
        grid=major,
        legend style={at={(1,0.3)},anchor=east},
    ]
    \addplot[thin] {sqrt(x)};
    \addlegendentry{$\sqrt{x}$}

    \addplot[dashdotted] coordinates {
        (0,0) (1,1) (2,1.414) (3,1.732) (4,2) (5,2.236) (6,2.449) (7,2.646) (8,2.828) (9,3) (10,3.162)
    };
    \addlegendentry{LB}
    
    \addplot[dashed] coordinates {
        (0,0+1/4) (1,1+1/4) (2,1.414+1/4) (3,1.732+1/4) (4,2+1/4) (5,2.236+1/4) (6,2.449+1/4) (7,2.646+1/4) (8,2.828+1/4) (9,3+1/4) (10,3.162+1/4)
    };
    \addlegendentry{UB}
    \end{axis}
\end{tikzpicture}
}
\caption{Piecewise linear lower and upper bounds for $\sqrt{x}$.}
\label{fig:ub_lb_sqrt}
\end{figure}
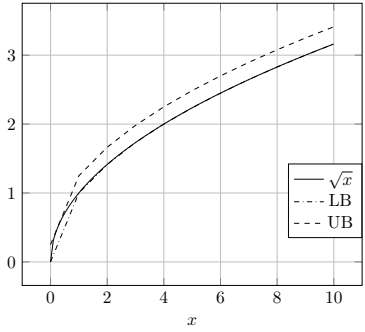
We implemented this constraint as shown in Fig. \ref{fig:piecewise_sqrt}. 
\begin{figure}
\begin{framed}
\begin{minipage}{\textwidth}
\scriptsize
\begin{verbatim}
dvar float+ V;   // knapsack weight variance
dvar float+ S;   // knapsack weight standard deviation
float+ s = ...;  // linearisation step
float+ x0 = ...; // value of the function at 0
range breakpoints = 1..Q/s; 
float slopes[i in breakpoints] = (sqrt(i*s) - sqrt((i-1)*s))/s; 
S == piecewise(b in breakpoints){slopes[b] -> b*s; 0}(0, x0)(V); 
\end{verbatim}
\end{minipage}
\end{framed}
\caption{Piecewise linear upper bound of function \texttt{sqrt} in IBM ILOG OPL, where $\texttt{Q}=\mathbf{1}'_n\sigma^2$, or $\texttt{Q}=\mathbf{1}'_n\Sigma \mathbf{1}_n$ in the multivariate normal case.}
\label{fig:piecewise_sqrt}
\end{figure}

\section{Proof of Lemma \ref{lemma:Amax}}\label{sec:lemma_Amax_proof}
Let $-\infty = z_0 < z_1 < \cdots < z_W = +\infty$ be the breakpoints of partition \( \{\Omega_k\}_{k=1}^W \), which is assumed to be contiguous and ordered on \( \mathbb{R} \);
then, \( \cup_{k=i}^W \Omega_k = [z_{i-1},\infty) \). Hence, with \( \varphi \) the standard normal pdf,
\[
A_i = \int_{z_{i-1}}^\infty z\,\varphi(z)\,dz
= \big[-\varphi(z)\big]_{z=z_{i-1}}^{z=\infty}
= \varphi(z_{i-1}),
\]
where we used \( \varphi'(z) = -z\,\varphi(z) \). Therefore
\[
A_{\max} = \max_{i=1,\dots,W} A_i
= \max_{i=1,\dots,W} \varphi(z_{i-1}),
\]
so \( A_{\max} \) is attained at the breakpoint \( z_{i-1} \) that is closest to \( 0 \), because \( \varphi \) is symmetric, strictly decreasing on \( [0,\infty) \), and maximized at \( 0 \).

If \( W \) is even and the partition is symmetric about \( 0 \) (as in the standard Jensen partitions used for the standard normal), then \( 0 \) is itself a breakpoint, so
\[
A_{\max} = \varphi(0) = \frac{1}{\sqrt{2\pi}} \approx 0.398942.
\]

If \( W \) is odd under a symmetric partition, then the two breakpoints closest to \( 0 \) are at \( \pm t_W \) with \( t_W > 0 \), hence
\[
A_{\max} = \varphi(t_W) < \varphi(0),
\]
but for the equal-probability (Jensen) partition one has \( t_W = \Phi^{-1}\!\big(\tfrac{m}{W}\big) \) with \( W = 2m+1 \), so \( t_W \to 0 \) as \( W \to \infty \), giving \( A_{\max} \uparrow \varphi(0) \).

This proves the claim. \hfill $\Box$

\section{Convergence of piecewise linear bounds for the first order loss function}\label{sec:convergence_pwl}

Let \(\mathcal{L}(y,\omega)\) be the first order loss of a standard
normal random variable \(\omega\) and let
\(\mathcal{L}_{\textup{lb}}(y,\omega)\) be its Jensen lower approximation obtained
from \(W\) partitions.  Define the pointwise error
\[
   e_W(y)\;\triangleq\;
     \mathcal{L}(y,\omega)-\mathcal{L}_{\textup{lb}}(y,\omega).
\]
We show that
\begin{equation}\label{eq:supErr}
  \sup_{y\in\mathbb{R}} e_W(y)
  \;=\;O\!\bigl(W^{-2}\bigr)
  \qquad\text{as } W\to\infty.
\end{equation}

Since \(\omega\sim\mathcal N(0,1)\),
\(\mathcal{L}(y,\omega)\) is twice
continuously differentiable with
\(\mathcal{L}''(y)=\varphi(y)\), where
\(\varphi(y)=\tfrac1{\sqrt{2\pi}}e^{-y^{2}/2}\)
is the standard normal density.  Hence
\(0<\mathcal{L}''(y)\le 1/\sqrt{2\pi}\) for all \(y\in\mathbb R\).

Let \(\Phi\) and \(\Phi^{-1}\) denote the standard normal cdf and
quantile function.  Partition the real line into \(W\) intervals of
equal probability
\[
   \Omega_i=[z_{i-1},z_i],
   \qquad
   z_i=\Phi^{-1}\!\bigl(i/W\bigr),
   \qquad
   i=0,\dots,W.
\]
For each interval set
\(\mu_i=\mbox{E}[\omega\mid\Omega_i]\) and
\(\Delta_i=\mu_{i+1}-\mu_i\).
Consider an arbitrary value \(\eta\) with \(0<\eta<\tfrac12\).
Because \(\Phi^{-1}\) is continuously differentiable on
\([\eta,1-\eta]\) with bounded derivative
\[
   0<\frac{d}{dp}\Phi^{-1}(p)=\frac{1}{\varphi(\Phi^{-1}(p))}
      \le \frac{1}{\varphi(\Phi^{-1}(\eta))}\;,
\]
there exists \(C_\eta>0\) such that
\(\lvert z_{i+1}-z_i\rvert\le C_\eta/W\) whenever
\(i/W\in[\eta,1-\eta]\).
Since \(\mu_i\in[z_{i-1},z_i]\), the same bound holds for the spacing
of conditional means:
\[
   \Delta\;\triangleq\;\max_{0\le i\le W-1}\Delta_i
   \;=\;O(W^{-1}).
\]
Note that the \(o(1/W)\) behaviour in the two tails
\(p<\eta\) or \(p>1-\eta\) is immaterial for the global rate.

Fix \(y\in[\mu_i,\mu_{i+1}]\).
By Taylor’s theorem there exists \(\xi\in[\mu_i,\mu_{i+1}]\)
such that
\[
   \mathcal{L}(y,\omega)
   \;=\;
   \mathcal{L}_{\textup{lb}}(y,\omega)
   \;+\;
   \frac12\,\mathcal{L}''(\xi)\,(y-\mu_i)^2
   +O\!\bigl((y-\mu_i)^3\bigr).
\]
Using \(\mathcal{L}''(\xi)=\varphi(\xi)\le 1/\sqrt{2\pi}\) and
\(\lvert y-\mu_i\rvert\le\Delta\) we obtain
\(e_W(y)\le(2\sqrt{2\pi})^{-1}\Delta^{2}\).

Since \(\Delta=O(W^{-1})\), the right-hand side is
\(O(W^{-2})\), and \eqref{eq:supErr} follows. 

\section{Numerical example of Section 5.1.2}\label{sec:numerical_example}
We consider the following SSKP instance: item weights are normally distributed with means $\bm{\mu}'$ and standard deviations $\bm{\sigma}'=c_v\,\bm{\mu}'$, where $c_v$ is a constant coefficient of variation; expected item values ($\mbox{E}[\bm{\varphi}']$) and mean item weights are shown in Table \ref{tab:data_instance}; finally, $c=10$; and $C=100$. In Table \ref{tab:cv_epsilon_wq_gap} we report values of $W$ and $Q$ for varying $c_v\in\{0.1,0.2,0.3\}$ and $\varepsilon\in\{10,1,0.1\}$. Both $W$ and $Q$ grow in line with the asymptotic analysis presented.
\begin{table}[ht]
\scriptsize
\centering
\begin{tabular}{c|cccccccccc}
Item $i$                & 1   & 2   & 3  & 4   & 5   & 6  & 7 & 8   & 9   & 10 \\ 
\midrule
$\operatorname{E}[\varphi_i]$ & 111 & 111 & 21 & 117 & 123 & 34 & 3 & 121 & 112 & 12 \\
$\mu_i$                       &  44 &  42 & 73 &  15 &  71 & 12 &13 &  14 &  23 & 15 \\
\end{tabular}
\caption{Expected item values and mean item weights for the test
instance}
\label{tab:data_instance}
\end{table}
\vspace{-1em}
\begin{table}[ht]
\scriptsize
\centering
\begin{tabular}{cc|ccc|ccc|ccc|}
\multicolumn{2}{c}{} & \multicolumn{9}{c}{$c_v$} \\
\multicolumn{2}{c}{} & \multicolumn{3}{c}{0.1} & \multicolumn{3}{c}{0.2} & \multicolumn{3}{c}{0.3} \\
\cline{3-11}
\multicolumn{2}{c|}{} & $W$ & $Q$ & gap & $W$ & $Q$ & gap & $W$ & $Q$ & gap \\
\cline{3-11}
\multirow{3}{*}{$\varepsilon$} 
& 10   & 5 & 8 & 5.08 & 7 & 27 & 5.87 & 8 & 56 & 6.81 \\
& 1    & 14 & 620 & 0.48 & 19 & 2500 & 0.44 & 23 & 5607 & 0.67 \\
& 0.1  & 41 & 62023 & 0.04 & 61 & 247859 & 0.07  & 90 & 557209 & 0.06  \\
\cline{3-11}
\end{tabular}
\caption{Table showing values of $W$, $Q$, and gap for combinations of $c_v$ and $\varepsilon$, where gap is defined as $Z^{\mathrm U}(W,Q)-Z^{\mathrm L}(W,Q)$}
\label{tab:cv_epsilon_wq_gap}
\end{table}

\section{SAA implementation}\label{sec:saa_implementation}
\vspace{-0.5em}
In this section, we assume the reader is familiar with the discussion and notation in \cite{Kleywegt2002}.
Let $S \subseteq \{0,1\}^n$ be the full decision space of candidate solutions (knapsack selections). With $N$ denoting the number of scenarios per SAA replication, we seek to include the SAA $\delta$-optimal set $\hat{S}_N^\delta$ in the true $\varepsilon$-optimal set $S^\varepsilon$ with confidence $1-\alpha$. By \citeauthor{Kleywegt2002}'s (2.16), this is ensured if $|S\setminus S^\varepsilon|\,e^{-N\gamma(\delta,\varepsilon)} \le \alpha$, where $\gamma(\delta,\varepsilon)$ is the (worst-case) LD rate for misclassification errors. Since $S^\varepsilon$ is unknown, we use the standard bound $|S\setminus S^\varepsilon|\le|S|\le 2^n$, which yields the sufficient condition $N \;\ge\; \frac{1}{\hat{\gamma}(\delta,\varepsilon)} \log\frac{|S|}{\alpha}$, cf.\ (2.22). 
We approximate the LD rate by the quadratic bound (2.20) and use the plug‑in estimator $\hat{\gamma}(\delta,\varepsilon) \;=\; \frac{(\varepsilon-\delta)^2}{3\,\hat{\sigma}_{\max}^2}$, where $\hat{\sigma}_{\max}^2$, c.f. (2.21), is the largest sample variance --- across all stored replications and candidate solutions compared to the current incumbent $\hat{x}$ --- of the per‑scenario difference in objective value computed within each replication's scenario set. Whenever the incumbent changes, $\hat{\sigma}_{\max}^2$ is recomputed against all stored replications. 
At each \(N\), we solve replications until we obtain at least one nonzero difference-variance to estimate \(\hat\sigma^2_{\max}\). We start from \(N_0=64\), double \(N\) until the LD condition is met, a wall-clock limit $T_{\text{lim}}$ is hit, or \(N\) reaches a cap \(N_{\text{cap}}\) determined as the value of $N$ for which solving one SAA replication takes longer than $T_{\text{lim}}/33$ --- the rationale for the denominator is that we need enough time to run ``warm-up'' replications during the certification phase. When increasing or capping \(N\), we reset and discard prior replications to recompute \(\hat{\sigma}^2_{\max}\) at the new \(N\).
Once $N$ is selected, we run SAA replications as follows: we first collect 32 ``warm-up'' replications to stabilise the gap estimates, then add replications one by one until the solution is certified, the time limit is reached, or we hit $M=1000$ replications. We set $\delta=0$, $\alpha=0.05$, and $\varepsilon=10^{-4}$. 
Certification is based on the first CLT-based optimality gap estimator $\hat g_{N'} (\hat x) - \bar v_N^M$  of \cite[][Section 3.3]{Kleywegt2002}. Each incumbent is evaluated out-of-sample with $N'=10^5$ Monte Carlo draws; the reported SAA runtime includes all SAA solves plus this out-of-sample evaluation.

\section{Performance of PWL}\label{sec:pivot_tables_pwl}
\vspace{-0.5em}
\paragraph{Normally distributed item weights}
A pivot table reporting average solution times for different instance sizes, instance types, and coefficients of variation considered is presented in Table \ref{tab:solution_time_sskp_normal_pwl}. 
\modified{In Table \ref{tab:solution_time_sskp_PWL_ablation} we illustrate average solution times when multithreading, dual presolve, and dynamic search are disabled.}
In Table \ref{tab:solution_time_sskp_LC_vs_LC_WS} we contrast the performance of LC and LC with warm start.
\vspace{-0.5em}
\paragraph{Multivariate normal distributed item weights} A pivot table reporting average solution times for different instance sizes, instance types, coefficients of variation, and correlation coefficients is presented in Table \ref{tab:solution_time_sskp_multinormal_pwl}.

\begin{table}[ht]
\centering
\begin{minipage}{0.48\textwidth}
\centering
\begin{tabular}{lllllll}
\toprule
 & $n$  & 25 & 50 & 100 & 200 & 500 \\
\midrule
\multirow{9}{*}{\rotatebox[origin=c]{90}{\em instance types}} & U & 2.74 & 3.79 & 5.42 & 6.06 & 9.36 \\
 & WC & 5.16 & 7.57 & 9.53 & 14.4 & 17.5 \\
 & SC & 6.34 & 18.0 & 35.4 & 35.4 & 26.5 \\
 & ISC & 8.53 & 11.2 & 19.2 & 108 & 24.4 \\
 & ASC & 7.46 & 17.0 & 22.6 & 25.1 & 25.8 \\
 & SS & 11.2 & 14.6 & 21.0 & 25.9 & 24.0 \\
 & USW & 5.53 & 7.71 & 11.9 & 16.2 & 29.2 \\
 & PC & 11.0 & 13.2 & 16.8 & 19.8 & 20.7 \\
 & C & 5.29 & 10.6 & 14.9 & 20.0 & 20.5 \\
\midrule
\multirow{2}{*}{\rotatebox[origin=c]{90}{$c_v$}} 
& 0.1 & 7.43 & 12.3 & 16.8 & 25.5 & 20.7 \\
 & 0.2 & 6.65 & 10.7 & 18.1 & 34.7 & 23.2 \\
\midrule
 &  & 7.04 & 11.5 & 17.4 & 30.1 & 22.0 \\
\bottomrule
\end{tabular}
\caption{Average solution time (s) of PWL for normally distributed item weights}
\label{tab:solution_time_sskp_normal_pwl}
\end{minipage}
\hfill
\begin{minipage}{0.48\textwidth}
\centering
\begin{tabular}{lllllll}
\toprule
 &  &  & 25 & 50 & 100 \\
\midrule
\multirow{9}{*}{\rotatebox[origin=c]{90}{\em instance types}} & U &  & 2.21 & 5.00 & 11.5 \\
 & WC &  & 3.87 & 5.79 & 11.6 \\
 & SC &  & 8.15 & 18.4 & 134 \\
 & ISC &  & 9.41 & 19.0 & 138 \\
 & ASC &  & 10.5 & 25.8 & 123 \\
 & SS &  & 14.1 & 32.6 & 46.7 \\
 & USW &  & 8.35 & 37.3 & 137 \\
 & PC &  & 10.2 & 20.5 & 25.9 \\
 & C &  & 6.79 & 14.1 & 46.2 \\
\midrule
\multirow{2}{*}{\rotatebox[origin=c]{90}{$c_v$}} & 0.1 &  & 7.43 & 16.9 & 59.3 \\
 & 0.2 &  & 8.96 & 22.7 & 91.0 \\
 \midrule
\multirow{2}{*}{\rotatebox[origin=c]{90}{$\rho$}} & 0.75 &  & 7.50 & 13.1 & 64.4 \\
& 0.95 &  & 8.90 & 26.5 & 85.9 \\
\midrule
& &  & 8.20 & 19.8 & 75.1 \\
\bottomrule
\end{tabular}
\caption{Average solution time (s) of PWL for multivariate normal distributed item weights}
\label{tab:solution_time_sskp_multinormal_pwl}
\end{minipage}
\end{table}

\begin{table}[ht]
\centering
\begin{minipage}{0.38\textwidth}
\centering
\modified{
\begin{tabular}{lllllll}
\toprule
 & $n$  & 25 & 50 \\
\midrule
\multirow{6}{*}{\rotatebox[origin=c]{90}{\em instance types}} & U & 11.1 & 105  \\
 & WC & 15.7 & 136 \\
 & SC & 64.6 & 439 \\
 & ISC & 220 & 439 \\
 & ASC & 94.2 & 472 \\
 & SS & 84.1 & 416 \\
 & USW & 160 & 392 \\
 & PC & 64.4 & 327 \\
 & C & 22.0 & 240 \\
\midrule
\multirow{2}{*}{\rotatebox[origin=c]{90}{$c_v$}} & 0.1 & 18.5 & 197 \\
 & 0.2 & 145 & 462 \\
\midrule
 &  & 81.9 & 330 \\
\bottomrule
\end{tabular}
\caption{Average solution time (s) of PWL for normally distributed item weights when multithreading, dual presolve, and dynamic search are disabled}
\label{tab:solution_time_sskp_PWL_ablation}
}
\end{minipage}
\hfill
\begin{minipage}{0.58\textwidth}
\centering
\resizebox{\textwidth}{!}{
\begin{tabular}{lllllllllllll}
\toprule
 & $n$ & \multicolumn{2}{l}{25} & \multicolumn{2}{l}{50} & \multicolumn{2}{l}{100} & \multicolumn{2}{l}{200} & \multicolumn{2}{l}{500} \\
 & Method & LC & WS & LC & WS & LC & WS & LC & WS & LC & WS \\
\midrule
\multirow{9}{*}{\rotatebox[origin=c]{90}{\em instance types}} 
 & U & 0.10 & 0.46 & 0.41 & 0.53 & 5.84 & 3.57 & 149 & 70.8 & 600 & 600 \\
 & WC & 0.02 & 0.51 & 0.28 & 1.14 & 5.04 & 8.84 & 104 & 138 & 600 & 600 \\
 & SC & 0.04 & 0.75 & 1.14 & 9.95 & 28.1 & 163 & 372 & 452 & 600 & 600 \\
 & ISC & 0.08 & 0.70 & 0.97 & 2.38 & 42.9 & 168 & 437 & 371 & 600 & 600 \\
 & ASC & 0.04 & 1.51 & 0.69 & 7.43 & 17.6 & 94.9 & 263 & 315 & 600 & 600 \\
 & SS & 0.06 & 1.07 & 0.72 & 2.99 & 17.5 & 18.2 & 285 & 241 & 600 & 600 \\
 & USW & 0.02 & 0.55 & 0.29 & 0.97 & 7.94 & 8.79 & 155 & 86.3 & 600 & 600 \\
 & PC & 0.04 & 1.73 & 0.66 & 1.82 & 10.8 & 10.8 & 181 & 135 & 600 & 600 \\
 & C & 0.04 & 0.62 & 0.51 & 6.03 & 4.98 & 19.3 & 110 & 170 & 600 & 600 \\
\midrule
\multirow{2}{*}{\rotatebox[origin=c]{90}{$c_v$}} 
& 0.1 & 0.06 & 1.00 & 0.69 & 5.07 & 19.9 & 42.1 & 225 & 226 & 600 & 600 \\
 & 0.2 & 0.04 & 0.76 & 0.57 & 2.32 & 11.4 & 68.2 & 231 & 214 & 600 & 600 \\
\bottomrule
\end{tabular}
}
\caption{Average solution time (s) for LC and LC with warm start (WS) using 25 piecewise linearisation segments under normally distributed item weights}
\label{tab:solution_time_sskp_LC_vs_LC_WS}
\end{minipage}
\end{table}

\section{Performance of KKT}\label{sec:boxplots_KKT}
\vspace{-0.5em}
In Figure~\ref{fig:pwl_vs_baselines_normal}, KKT results appear to cluster around certain runtimes. We further investigate this matter in Figure~\ref{fig:kkt_runtimes}. KKT runtimes are almost entirely determined by the instance ``type'' considered (U, WC, SC, ISC, ASC, SS, USW, PC, C) and by the number of items; when these two factors are fixed, runtime variability is minimal among instances. Moreover, SS instances --- and to a certain extent USW instances --- appear to be easier than others for KKT; conversely, all remaining instance types are equally difficult.

\begin{figure}
\centering
\includegraphics[width=0.95\textwidth, angle=0]{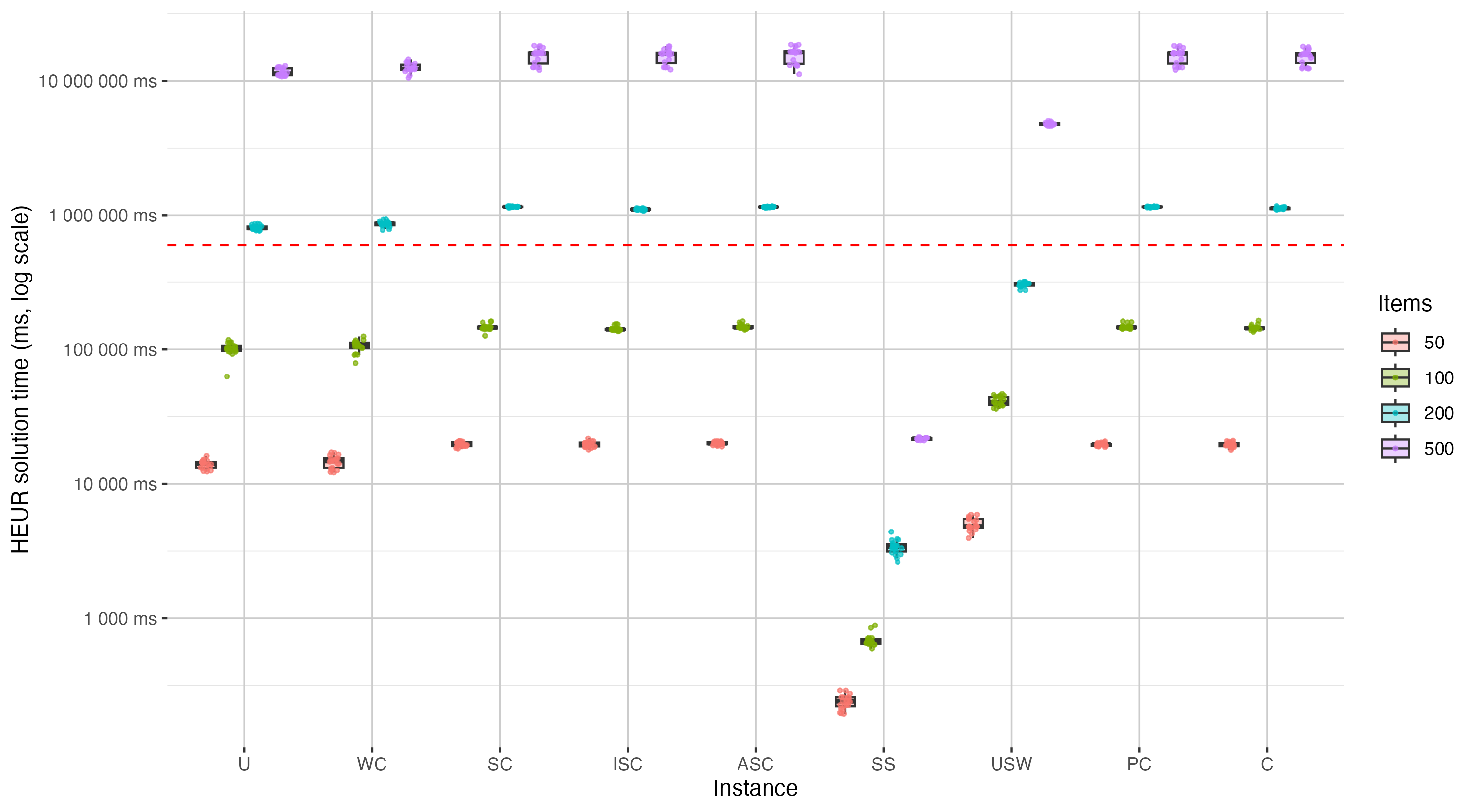}
\caption{Runtimes for KKT by instance type and item number; the dashed red line represents the time limit of 10 minutes.}
\label{fig:kkt_runtimes}
\end{figure}

Note that this is not an artefact of our benchmark; it is a direct consequence of how the KKT-based heuristic enumerates its candidate service-level breakpoints $ z $ and thresholds. Following notation from \cite[][Sections 3.2 and 3.3]{Merzifonluolu2011}, for a fixed item count $ n $ and instance ``type,'' the algorithm ends up generating (almost) a fixed number of $ z $-values and threshold positions, so the total amount of work is nearly constant --- hence the tight runtime clusters. Two types stand out. \textbf{SS (subset-sum)} is especially fast because all pairwise equalities $ \rho_i(z)=\rho_j(z) $ collapse to the same $ z $, so $ \lvert \mathcal{Z} \rvert $ is $ O(1) $ rather than $ \Theta(n^2) $. \textbf{USW (uncorrelated with similar weights)} is also faster because many pairwise candidates are discarded --- near-degenerate denominators or ratios outside $ (0,1) $ --- so $ \lvert \mathcal{Z} \rvert $ is small.
Everything else produces many valid, distinct $ z(i,j) $, so $ \lvert \mathcal{Z} \rvert \approx \Theta(n^2) $. Since the code generates $ \approx n $ candidates per $ z $, runtime scales like $ \Theta(n\cdot \lvert \mathcal{Z} \rvert)\approx \Theta(n^3) $ with a type-dependent constant, giving those horizontal clusters per $ n $ / type.

\section{Performance of LC}\label{sec:pivot_tables_lc}

\paragraph{Item weights following a generic distribution}
Pivot tables reporting average solution times for different instance sizes, instance types, and coefficients of variation considered are presented in Table \ref{tab:solution_time_sskp_gamma_LC} (gamma distributed item weights) and Table \ref{tab:solution_time_sskp_lognormal_LC} (lognormal distributed item weights). 

\begin{table}[ht]
\centering
\begin{minipage}{0.48\textwidth}
\centering
\begin{tabular}{lllllll}
\toprule
 & $n$  &  & 25 & 50 \\
\midrule
\multirow{9}{*}{\rotatebox[origin=c]{90}{\em instance types}} 
 & U &  & 29.9 & 234 \\
 & WC &  & 32.5 & 212 \\
 & SC &  & 279 & 586 \\
 & ISC &  & 428 & 583 \\
 & ASC &  & 298 & 601 \\
 & SS &  & 594 & 600 \\
 & USW &  & 50.3 & 193 \\
 & PC &  & 552 & 600 \\
 & C &  & 94.9 & 444 \\
\midrule
\multirow{2}{*}{\rotatebox[origin=c]{90}{$c_v$}} 
& 0.1 &  & 247 & 436 \\
 & 0.2 &  & 276 & 465 \\
\midrule
 & &  & 262 & 450 \\
\bottomrule
\end{tabular}
\caption{Average solution time (s) of LC for gamma distributed item weights}
\label{tab:solution_time_sskp_gamma_LC}
\end{minipage}
\hfill
\begin{minipage}{0.48\textwidth}
\centering
\begin{tabular}{lllllll}
\toprule
 & $n$  &  & 25 & 50 \\
\midrule
\multirow{9}{*}{\rotatebox[origin=c]{90}{\em instance types}} 
 & U &  & 2.00 & 9.48 \\
 & WC &  & 3.03 & 10 \\
 & SC &  & 12.2 & 376 \\
 & ISC &  & 51.2 & 415 \\
 & ASC &  & 30.2 & 417 \\
 & SS &  & 141 & 585 \\
 & USW &  & 5.67 & 12.0 \\
 & PC &  & 53.3 & 536 \\
 & C &  & 7.26 & 83.3 \\
\midrule
\multirow{2}{*}{\rotatebox[origin=c]{90}{$c_v$}} 
 & 0.1 &  & 27.5 & 245 \\
 & 0.2 &  & 40.6 & 297 \\
\midrule
& & & 34.1 & 271 \\
\bottomrule
\end{tabular}
\caption{Average solution time (s) of LC for lognormal distributed item weights}
\label{tab:solution_time_sskp_lognormal_LC}
\end{minipage}
\end{table}

\section{Computational time and optimality gap analysis for PWL under multivariate normal distributed item weights}\label{sec:analysis_opt_gap_comp_time_mvn}
For the 1080 instances considered under {\em normally distributed item weights} computational times and optimality gaps for varying $W$ are shown in Figure~\ref{fig:pwl_comp_time_segs_mvn} and Figure~\ref{fig:pwl_rel_opt_gap_mvn}, respectively.

\begin{figure}
\centering
\includegraphics[width=\textwidth, angle=0]{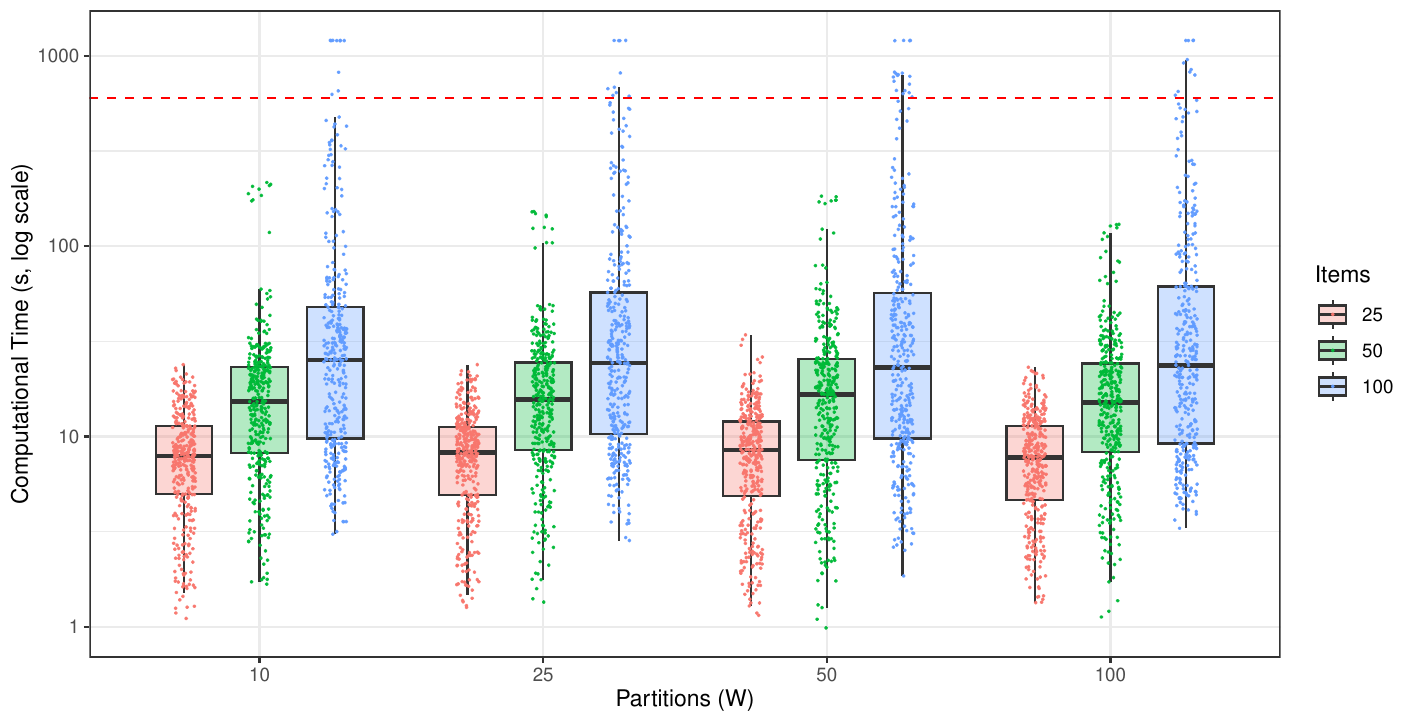}
\caption{Boxplot with jitter illustrating runtimes for PWL under multivariate normal distributed item weights and varying number of partitions ($W$); the dashed red line represents the time limit of 10 minutes.}
\label{fig:pwl_comp_time_segs_mvn}
\end{figure}
\begin{figure}
\centering
\includegraphics[width=0.9\textwidth, angle=0]{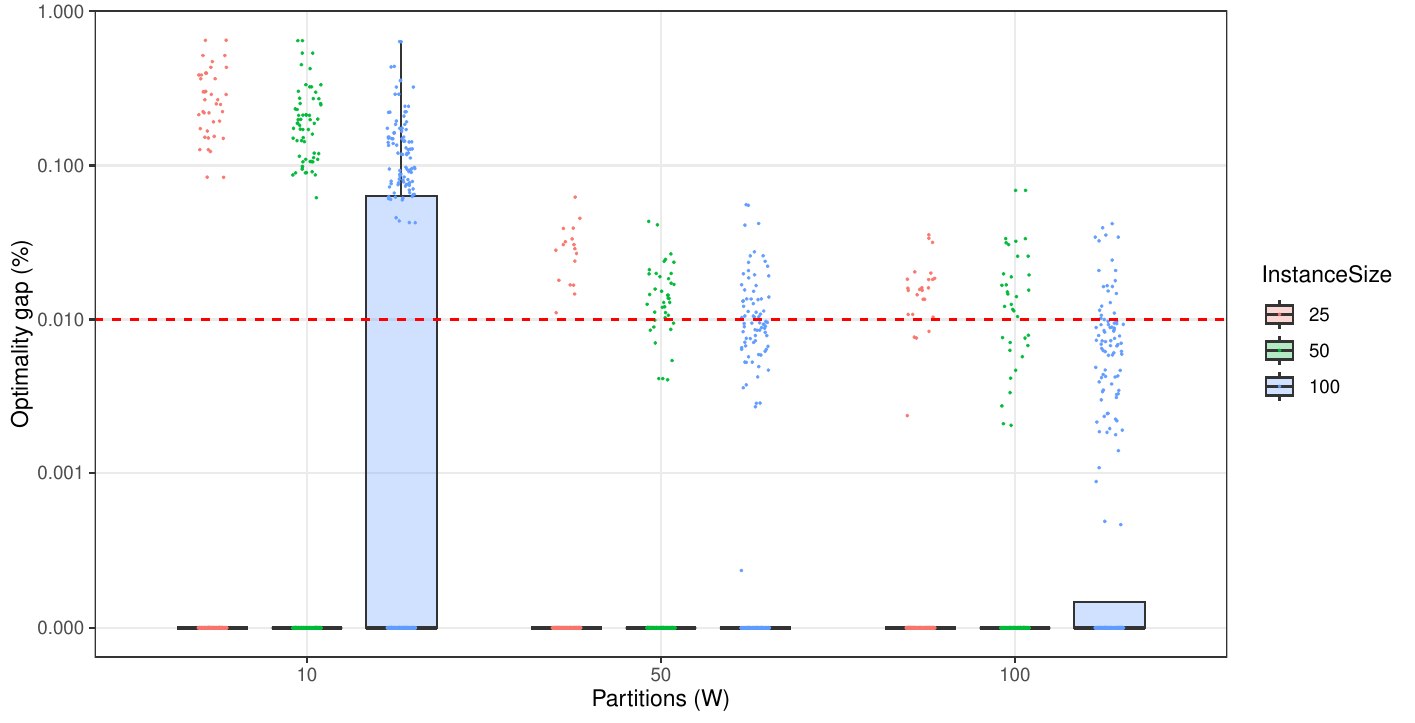}
\caption{Boxplot with jitter illustrating optimality gaps for PWL under multivariate normal distributed item weights and varying number of partitions ($W$); the dashed red line represents the target tolerance.}
\label{fig:pwl_rel_opt_gap_mvn}
\end{figure}

\section{Statistical certificates}
\label{sec:statistical_certificates_RH}

For the 180 instances considered under normally distributed item weights, in Figure \ref{fig:statistical_certificates_RH} we illustrate by means of boxplots with jitter optimality gaps for the receding horizon (RH) heuristics presented in Section \ref{sec:receding_horizon} computed against the optimal policy obtained via stochastic dynamic programming (RH) and those obtained by means of the statistical certificates produced by the approach in Section \ref{sec:rh_certificate} (RH STAT).

\begin{figure}
\centering
\includegraphics[width=0.75\textwidth, angle=0]{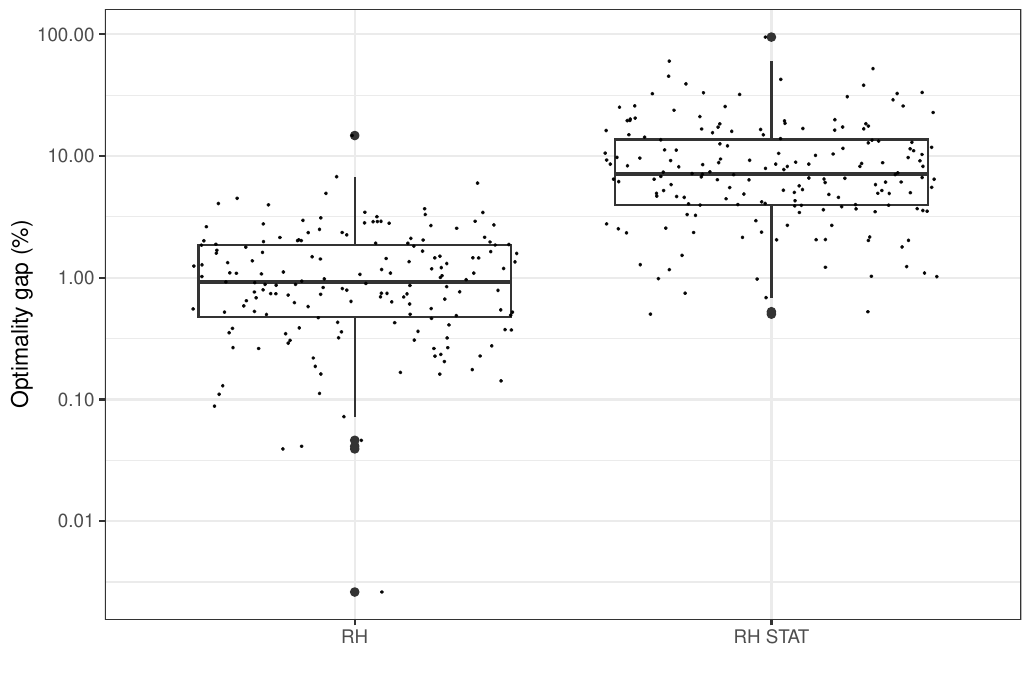}
\caption{Boxplot with jitter illustrating optimality gaps for the receding horizon heuristic computed against the optimal policy obtained via stochastic dynamic programming (RH) and those obtained by means of the statistical certificates produced by the approach in Section \ref{sec:rh_certificate} (RH STAT) for instances under normally distributed item weights.}
\label{fig:statistical_certificates_RH}
\end{figure}

\section*{CRediT author statement} {\bf Roberto Rossi}: Conceptualization, Methodology, Software, Data curation, Writing -- original draft;
{\bf Steven Prestwich}: Funding acquisition, Validation, Writing -- review and editing;
{\bf S. Armagan Tarim}: Software, Validation, Writing -- review and editing.


\bibliography{skp-bibliography}

\end{document}